\documentclass[a4paper,11pt]{amsart}

\usepackage{amssymb,amsmath,amsthm,mathrsfs,enumerate,graphicx, color}
\usepackage[pdfpagelabels,colorlinks,linkcolor=blue,citecolor=black,urlcolor=blue]{hyperref}

\usepackage[british]{babel}

\usepackage{esint}
\usepackage{tikz}
\usetikzlibrary{arrows}
\usepackage{amsmath, amssymb}
\usepackage{enumitem}
\usepackage{xcolor}
\usepackage[T1]{fontenc}
\usepackage[utf8]{inputenc}

\usepackage{enumitem}
\newlist{myassump}{enumerate}{1}
\setlist[myassump]{label=\textbf{(A\arabic*)}, ref=\textbf{(A\arabic*)}}

\usepackage{tgtermes}

\usepackage{titlesec}

\titleformat{\section}
{\normalfont\bfseries\Large}  % Format: normal font, bold, large size
{\thesection}{1em}{}          % Section number formatting
\titleformat{\subsection}
{\normalfont\bfseries\large}  % Format: normal font, bold, large size
{\thesubsection}{1em}{}          % Section number formatting

\usepackage{cite}
\makeatletter
\def\Ddots{\mathinner{\mkern1mu\raise\p@
		\vbox{\kern7\p@\hbox{.}}\mkern2mu
		\raise4\p@\hbox{.}\mkern2mu\raise7\p@\hbox{.}\mkern1mu}}
\makeatother

\def\XXint#1#2#3{{\setbox0=\hbox{$#1{#2#3}{\int}$}
		\vcenter{\hbox{$#2#3$}}\kern-.5\wd0}}

\newcommand{\vertiii}[1]{{\left\vert\kern-0.25ex\left\vert\kern-0.25ex\left\vert #1
		\right\vert\kern-0.25ex\right\vert\kern-0.25ex\right\vert}}

\newtheorem{theorem}{Theorem}[section]

\theoremstyle{plain}

\theoremstyle{plain}

\newtheorem{lemma}[theorem]{Lemma}

\newtheorem{remark}[theorem]{Remark}

\theoremstyle{definition}
\newcommand{\f}{\frac}

\newcommand{\Rn}{\mathbb{R}^n}

\definecolor{rojo}{rgb}{1,0,0}
\definecolor{newrojo}{rgb}{1,0,0}
\definecolor{verde}{cmyk}{0.92,0,0.59,0.25}

\def\bey{\begin{eqnarray*}}
	\def\eey{\end{eqnarray*}}

\allowdisplaybreaks

\begin{document}
	
	\title[Endpoint estimates for certain singular integrals]{Endpoint estimates for certain singular integrals\\
	with non-smooth kernels}
	
	\author[X. Han]{Xueting Han}
	\author[X. Huo]{Xuejing Huo$^*$}
	\address[Xueting Han]{Department of Mathematics, Hefei Institute of Technology, Hefei 238706, Anhui, China}
	\email{hanxueting12@163.com}
	\address[Xuejing Huo]{School of Mathematical and Physical Sciences, Macquarie University, Sydney 2109, NSW, Australia}
	\email{xuejing.huo@hdr.mq.edu.au}

	\thanks{$^*$Corresponding author.}
	
	%\subjclass{Primary: 42B25. Secondary: 43A85.}
	%\keywords{Muckenhoupt weights, BMO}

	\maketitle

	\section*{Abstract}

	Let \( L \) be a closed, densely defined operator of type \( \omega \) on \( L^2(\mathbb{R}^n) \)  with \( 0 \leq \omega < \pi/2 \). We assume that \( L \) possesses a bounded \( H_\infty \)-functional calculus and that its heat kernel satisfies suitable upper bounds. In this paper, we establish the boundedness from Lorentz spaces \( L^{p_0,1}(\mathbb{R}^n) \) to \( L^{p_0,\infty}(\mathbb{R}^n) \) 
    for some singular integrals associated with \( L \), including the vertical square function and the functional calculus of Laplace transform type, where \(p_0\) is determined by the upper bound of the heat kernel. As concrete applications, we obtain the endpoint estimates for the above singular integrals associated with both the Hardy operator and the Kolmogorov operator.

	\section{Introduction and Main Results}
	
	In this paper, we study new endpoint estimates for two types of singular integrals associated with an operator \(L\), without assuming any regularity of the heat kernel of \(L\). More precisely, let \(L\) be a closed, densely defined operator of type \(\omega\) on \(L^{2}(\mathbb R^n)\), with \(0\leq \omega<\pi/2\). By the Hille--Yoshida Theorem, $L$ generates a holomorphic semigroup $e^{-zL}$ on the sector \(\{z\in \mathbb{C}: z\neq 0,\ |\arg(z)|<\pi/2-\omega\}\), and we denote the heat kernel of \(e^{-zL}\) by \(p_{z}(x,y)\).
    %for \((x,y)\in \mathbb{R}^n \times \mathbb{R}^n\). 
    Throughout the paper, we assume that \(L\) satisfies the following two assumptions:
	\begin{myassump}
		\item \label{A1CONDI} \(L\) has a bounded $H_\infty$--functional calculus on \(L^2(\mathbb R^n)\);
		\item \label{A2CONDI} The heat kernel \(p_t(x,y)\) satisfies that for $t>0$ and $x,y\in \mathbb R^n\backslash\{0\}$, there exist $\alpha \in (0,2)$, $\epsilon>0$, and $\theta, \sigma \ge 0$ with $n/(n-\sigma)<2<n/\theta$ such that
		\begin{equation}\label{eq:heat-kernel-bounds0}
			|p_t(x,y)| \lesssim t^{-n/\alpha}
			\Big(\frac{t^{1/\alpha}+|x-y|}{t^{1/\alpha}}\Big)^{-n-\epsilon}
			\Big(1+\frac{t^{1/\alpha}}{|x|}\Big)^\theta
			\Big(1+\frac{t^{1/\alpha}}{|y|}\Big)^\sigma.
		\end{equation}
	\end{myassump}
	
	We now present two examples of operators \(L\) satisfying assumptions \(\ref{A1CONDI}\) and \(\ref{A2CONDI}\). The Kolmogorov operator is defined by
	\begin{align*}
		\Lambda_{\kappa}:=(-\Delta)^{\beta/2} +\frac{\kappa}{|x|^{\beta}} x\cdot \nabla,
	\end{align*}
	for \(\beta\in(1,2]\) with \(\beta<(n+2)/2\), and a coupling constant \(\kappa\in \mathbb R\). The precise values of $\sigma$ and $\theta$ appearing in \ref{A2CONDI} can be found in \cite{BDM}.
	
	The Hardy operator is given by
	\begin{align*}
		L_{a}:=(-\Delta)^{\gamma/2} +\frac{a}{|x|^{\gamma}} \quad \text{in } L^{2}(\mathbb R^{n})
	\end{align*}
    with \(0<\gamma\le\min\{2,n\}\) and \(a\ge -\frac{2^{\gamma}\Gamma((n+\gamma)/4)^2}{\Gamma((n-\gamma)/4)^2}\),
	that is, the fractional Laplacian plus the scalar-valued, so-called Hardy potential \(a/|x|^{\gamma}\). See \cite{BD} for the corresponding values of $\sigma$ and $\theta$ in \ref{A2CONDI}. For further details concerning \(L\), we refer the reader to Section \ref{Pre} and to \cite{ADM, CDMY, DMc}.
	
	In this paper, we first consider the square function \(S_{L}\) associated with the operator \(L\), defined by
	\begin{align*}
		S_{L}f(x) =\Big(\int_{0}^{\infty} |tLe^{-tL}f(x)|^2 \frac{dt}{t}\Big)^{1/2}.
	\end{align*}

	 The assumption \ref{A1CONDI} implies that \(S_L\) is bounded on \(L^2(\mathbb R^n)\); see \cite{Mc}. We give a brief overview of related research on endpoint estimates for the operator \(S_L\) under various choices of \(L\) and assumptions on the corresponding heat kernel. 
     If \(L\) is either \(-\Delta\) or \(\sqrt{-\Delta}\), where \(\Delta\) is the Laplacian on \(\mathbb R^n\), then \(S_L\) corresponds to the classical Littlewood--Paley--Stein functions:
	\[
	g(f)(x) = \Big(\int_{0}^{\infty} t|\sqrt{-\Delta} e^{-t\sqrt{-\Delta}}f(x)|^2dt\Big)^{1/2}\,
    \text{and}\quad 
	h(f)(x) = \Big(\int_{0}^{\infty} t|\Delta e^{t\Delta}f(x)|^2dt\Big)^{1/2}.
	\]
    It is well known that these two square functions are of weak type \((1,1)\) and bounded on \(L^p(\mathbb R^n)\) for \(1<p<\infty\). 
	If \(L\) is a (non-negative) Laplace--Beltrami operator on a complete non-compact Riemannian manifold \(M\), Coulhon, Duong and Li \cite{CDL} obtained the weak type \((1,1)\) estimate for the square function under the assumption that the heat kernel of \(L\) satisfies Gaussian upper bounds, which play a crucial role in their proofs. In addition, if \(L\) is a second-order elliptic operator in divergence form, 
	the weak type \((1,1)\) estimate for the square function associated with \(L\) was also established, provided that the heat kernel of \(L\) satisfies either Gaussian upper bounds or Poisson-type upper bounds.

   However, as shown in \cite{AT2}, Gaussian upper bounds for the kernel of the semigroup \(e^{-tL}\) with a divergence form operator \(L\) do not always hold: they hold in dimensions \(n=1,2\), but may fail in dimensions \(n \geq 5\). 
   %  in dimensions $n=1, 2$, but in dimension $n \geq 5$, these bounds may fail. 
   Blunck and Kunstmann \cite{BK} established a weak type \((p,p)\) criterion for non-integral operators for \(1 < p < \infty\). Their result generalises \cite[Theorem 6]{DM}, since the pointwise estimate of the heat kernel is not required. Yan \cite{Y} further
    %builds on the idea from \cite{HM} of proving \(L^p\) boundedness for the Riesz transforms, 
    %together with some techniques in \cite{BK}, to 
    obtained the weak type $(p_n,p_n)$ estimate for the generalised vertical Littlewood--Paley function associated with the divergence form operator \(L\),  where $p_n:=2n/(n+2)<2$ and $n \geq 3$, removing the Gaussian upper bound assumption on the kernel of \(e^{-tL}\). %removing the condition of Gaussian upper bound on \(L\).

%Let \(p_0=n/(n-\sigma)\) and \(q_0=n/\theta\). In our setting, the upper bound for the kernel of \(e^{-tL}\) is weaker than the Gaussian upper bound. Using the approaches in \cite{A, BK, Y}, one can prove that \(S_L\) is bounded on \(L^{p}(\mathbb R^n)\) for \(p_0<p<q_0\). However, this method does not yield the weak-type \((p_0,p_0)\), since the semigroup \(e^{-tL}\) does not satisfy the desired off-diagonal estimates detailed in \cite{A, BK, Y}. Hence, the main aim in this paper is to investigate an another endpoint estimate when \(p=p_0\).
     
     Let \(p_0=n/(n-\sigma)\) and \(q_0=n/\theta\). In our setting, the upper bound for the kernel of \(e^{-tL}\) is weaker than the Gaussian upper bound. Using the approaches in \cite{A, BK, Y}, one can prove that \(S_L\) is bounded on \(L^{p}(\mathbb R^n)\) for \(p_0<p<q_0\). However,  when \(p=p_0\), these methods do not yield the classical weak-type \((p_0,p_0)\) boundedness, since the semigroup \(e^{-tL}\) does not satisfy the desired off-diagonal estimates detailed in \cite{A, BK, Y}. Therefore, the main aim of this paper is to investigate an endpoint estimate at \(p=p_0\). 
     
Our first result in this paper is presented below.
	\begin{theorem}\label{thm: square function}
		Assume that $L$ satisfies the assumptions \ref{A1CONDI} and \ref{A2CONDI}. Then we have that $S_L$ is bounded from $L^{p_0,1}(\mathbb R^n)$ to $L^{p_0,\infty}(\mathbb R^n)$, i.e., for any $\lambda>0$,
		\begin{align*}
						|\{x\in \mathbb{R}^n: |S_L(f)(x)|>\lambda \}| \lesssim \lambda^{-p_0} \|f\|_{L^{p_0,1}(\mathbb{R}^n)}^{p_0},
		\end{align*}
		where, in the following, \(L^{p_0,1}(\mathbb R^n)\) denotes the Lorentz space and it is a proper subspace of \(L^{p_0}(\mathbb R^n)\). See Section \ref{Pre} for the definition of Lorentz spaces. 
	\end{theorem}

 %\begin{remark}
% \normalfont
%Theorem \ref{thm: square function} establishes the weak-type boundedness for the square function \(S_L\) on the Lorentz space \(L^{p_0,1}(\mathbb R^n)\), which provides an endpoint estimate at \(p=p_0\).
%Recall that the \(L^p\) boundedness for \(S_L\) with \(p_0<p<q_0\) can be obtained by the approaches in \cite{A, BK, Y}, but the literature cannot answer the endpoint case when \(p=p_0\). 
%      Theorem \ref{thm: square function} provides the weak-type estimates in Lorentz space at \(p=p_0\), whereas the existing approaches in \cite{A, BK, Y} may fail to yield the classical weak-type \((p_0,p_0)\) bound.
 %\end{remark}     

	The second class of singular integrals we focus on is the functional calculus of Laplace transform type $\mathcal M(L)$ associated with $L$. Let $m: [0,\infty) \to \mathbb C$ be a bounded function. Then $\mathcal M(L)$ is defined by
	\[
	\mathcal M(L) = \int_0^\infty m(t)Le^{-tL}dt.
	\]
	The spectral theory implies that $\mathcal M(L)$ is $L^2$-bounded. When \(m(t)=\frac{1}{\Gamma(1+i\alpha)}t^{i\alpha}\) for \(\alpha\in \mathbb{R}\), the operator \(\mathcal{M}(L)\) coincides with the imaginary power operator \(L^{i\alpha}\), where
\[
L^{i\alpha}=\frac{1}{\Gamma(1+i\alpha)}\int_{0}^{\infty} t^{i\alpha} L e^{-tL}\,dt, \quad i^2=-1.
\]

    In this paper, we investigate the endpoint estimates for $\mathcal M(L)$ associated with any $L$ satisfying assumptions \ref{A1CONDI} and \ref{A2CONDI} by establishing the boundedness from the Lorentz space $L^{p_0,1}(\mathbb{R}^n)$ to $L^{p_0,\infty}(\mathbb{R}^n)$. To be more explicit, we state our second result as follows. %The corresponding result is stated below.
	
	\begin{theorem}\label{THEMML}
		Assume that $L$ satisfies the assumptions \ref{A1CONDI} and \ref{A2CONDI}. Then we have that $\mathcal M(L)$ is bounded from $L^{p_0,1}(\mathbb{R}^n)$ to $L^{p_0,\infty}(\mathbb{R}^n)$, that is, for any $\lambda>0$,
		\begin{align*}
			|\{x\in \mathbb{R}^n: |\mathcal{M}(L)(f)(x)|>\lambda \}| \lesssim \lambda^{-p_0} \|f\|_{L^{p_0,1}(\mathbb{R}^n)}^{p_0}.
		\end{align*}
	\end{theorem}

\begin{remark}
\normalfont
The \(L^p\) boundedness of \(\mathcal{M}(L)\) for \(p_0 < p < q_0\) can be obtained from \cite{BK}, but the endpoint case \(p = p_0\) is not included. Theorem~\ref{THEMML}, together with the \(L^2\) boundedness of \(\mathcal{M}(L)\) implies that \(\mathcal{M}(L)\) is bounded on \(L^p\) for \(p_0 < p \leq 2\) by interpolation. By duality, we further obtain the boundedness on \(L^p\) for \(2 < p < q_0\). 
%Theorem~\ref{THEMML} provides the endpoint estimate at \(p = p_0\). Moreover, 

%Then \((\mathcal M(L))^{*}=\mathcal M(L^{*})\). Hence \(\mathcal M(L)\) is bounded on \(L^{p}\) for \(p_0<p<q_0\), since \(\mathcal M(L^{*})\) is bounded on \(L^{p}\) for \(p_0<p\leq 2\). An alternative approach to establishing the \(L^{p}\) boundedness of \(\mathcal M(L)\) for \(p_0<p<q_0\) can be found in \cite{A} and \cite{Y}.
\end{remark}
    
We emphasise that we adapt the techniques in \cite{L} with necessary modifications, leading to endpoint estimates in our framework. The approach in \cite{L} used the finite speed propagation property of the Schr\"odinger operator with the inverse-square potential. However, it is not directly applicable to establishing the desired results for \(S_L\) and \(\mathcal M(L)\) in this paper, since 
%\(L\) does not satisfy the Davies--Gaffney estimates, which are equivalent to 
\(L\) does not satisfy the finite speed propagation property.
%; see \cite{CS}.
%although Lai \cite{L} and we both obtain the endpoint estimate \(L^{p_0,1}\rightarrow L^{p_0,\infty}\), the approaches and the operators we studied are different. Lai's approaches in \cite{L} are related to the finite speed propagation property of the Schr\"odinger operator with the inverse-square potential.
%\(\mathcal L_a\). %We combines the ideas from \cite{L} and \cite{Y}, 

We primarily exploit the precise heat kernel upper bound of \(L\) in \ref{A2CONDI} to derive key \(L^{p_0,1}\rightarrow L^{p}\)  estimates for operators related to the semigroup \(e^{-tL}\), thereby bypassing the use of the finite speed propagation property. As applications, we obtain the endpoint estimates for the vertical square function and the functional calculus of Laplace transform type associated with both Kolmogorov operators \(\Lambda_\kappa\) and Hardy operators \(L_{a}\).

	The remainder of this paper is organised as follows. In Section~\ref{Pre}, we present several lemmas that will be used frequently throughout the paper. The proofs of Theorems~\ref{thm: square function} and~\ref{THEMML} are given in Sections~\ref{section: square operators} and~\ref{sec proof of ML}, respectively.
	
	Throughout this paper, $``C"$ denotes a positive constant, which is independent of the essential variables, and may vary across different occurrences. We write $X \lesssim Y$ to indicate the existence of a constant $C$ such that $X \leq CY$. If we write $X \approx Y$, then both $X \lesssim Y$ and $Y \lesssim X$ hold.

	\section{Preliminaries}\label{Pre}
	
	For constants $1 \leq p<\infty, 1 \leq q \leq \infty$, the Lorentz space $L^{p, q}(\mathbb{R}^n)$ is defined as the subset of measurable function space on $\mathbb{R}^n$ equipped with the norm:
	
	$$
	\|f\|_{L^{p, q}(\mathbb{R}^n)}= \begin{cases}\left(\int_0^{\infty}\left(t^{1/p} f^*(t)\right)^q \frac{d t}{t}\right)^{1/q} & \text { if } q<\infty, \\ \sup _{t>0} t^{1/p} f^*(t) & \text { if } q=\infty ,\end{cases}
	$$
	where \(f^*\) is the decreasing
	rearrangement of the function \(f\).
	
	By the above definitions, it is direct to observe that $L^{p, p}(\mathbb{R}^n)$ and $L^{\infty, \infty}(\mathbb{R}^n)$ coincide with the spaces $L^p(\mathbb{R}^n)$ and $L^\infty(\mathbb{R}^n)$, respectively. $L^{p, \infty}(\mathbb{R}^n)$ denotes the weak $L^{p}(\mathbb{R}^n)$ space. Moreover, the definition implies that \( \|\chi_E\|_{L^{p,1}(\mathbb{R}^n)}=\|\chi_E\|_{L^{p}(\mathbb{R}^n)} \) for any measurable set \(E\subset \mathbb{R}^n\), where \(\chi_E\) denotes the characteristic function of the set \(E\).

	It is well-known that the H\"older's inequality also holds for Lorentz space, for $1<p\leq \infty$,
	\begin{equation*}
		\int_{\mathbb{R}^n} |f(x)||g(x)|dx\leq
		\|f\|_{L^{p, \infty}(\mathbb{R}^n)}
		\|g\|_{L^{p', 1}(\mathbb{R}^n)},
	\end{equation*}
	where the conjugate index of \(\infty\) is \(1\); see, for example, \cite{KS2010}.

	For any locally integrable function \(f\) and measurable set \(E\subset \mathbb R^n\), we write the average of the function \(f\) as
	\[\fint_{E}|f|=\frac{1}{|E|}\int_{E}|f|.\]
	
	In this paper, we consider the Hardy--Littlewood \(s\)-maximal operator \(M_{s}\), \(1\leq s<\infty\), which is defined by
	\begin{align*}
		M_{s}f(x)&=\sup_{Q\ni x} \Big( \fint_{Q}|f(y)|^sdy
		\Big)^{1/s},	
	\end{align*}
	where \(Q\) denotes the cube in \(\mathbb R^n\).
	When \(s=1\), \(M_1\) is the Hardy--Littlewood maximal function $M$.
	Note that  $M$ is bounded on $L^p(\mathbb{R}^n)$ for \(1<p<\infty\) (see \cite{Gr1}).
	Then, it is obvious from the boundedness of \(M\) that \(M_s\) is \(L^p\)-bounded for any \(s<p<\infty\).

	Let \(p_{k,t}(x,y)\) denote the kernel of the operator \((tL)^{k}e^{-tL} \) for $k\geq 0$ and $t>0$. Then the following estimate holds for \(p_{k,t}(x,y)\).  A similar statement and its proof can be found in \cite{BDM}; We omit the proof. 
	\begin{lemma}\label{lemma: kernel of p_k,t}
		%\textup{(\cite{BDM})}
		For each $k\in \mathbb N =\{0,1,\ldots\}$ and $\tilde \epsilon \in (0,\epsilon)$, $\epsilon>0$, we have
		\begin{equation}\label{eq: heat_kernel p_{k,t}}
			|p_{k,t}(x,y)|\lesssim t^{-n/\alpha}\Big(\f{t^{1/\alpha}+|x-y|}{t^{1/\alpha}}\Big)^{-n-\tilde\epsilon} \Big(1+\f{t^{1/\alpha}}{|x|}\Big)^\theta\Big(1+\f{t^{1/\alpha}}{|y|}\Big)^\sigma.
		\end{equation}
	\end{lemma}
%\begin{remark}
%\normalfont
   
%\end{remark}

	For \(k=0\), we denote \(p_t(x,y)\) instead of \(p_{0,t}(x,y)\).
	For the kernel \(p_{1,t}(x,y)\) of the operator \(tLe^{-tL}\), we write
	\begin{align*}
		S_{L}f(x) =\Big(\int_{0}^{\infty} |tLe^{-tL}f(x)|^2 \frac{dt}{t}\Big)^{1/2} = \Big( \int_{0}^{\infty} \big| \int_{\mathbb R^n} p_{1,t}(x,y) f(y)dy \big|^2  \frac{dt}{t}\Big)^{1/2}.
	\end{align*}

	For $t>0$, let $T_t(x,y)$ be the function satisfying the estimates as in \eqref{eq:heat-kernel-bounds0} and $T_t$ be the operator defined by
	\[
	T_tf(x) = \int_{\Rn} T_t(x,y)f(y)dy.
	\]

Recall that \(p_0:=n/(n-\sigma)\) and \(q_0:=n/\theta\).
    Let \(p,\ q\) be any real numbers with \(p_0<p\leq q<q_0\). %\[p_0:=\frac{n}{n-\sigma}<p\leq q<\frac{n}{\theta}=:q_0.\]
Then, we present some crucial lemmas in the following which will be used in the proofs of the theorems.

	\begin{lemma}\label{lem: T_t: L^{p_0,1} to L^q}  For a cube $Q$  with side length $r_Q$, denote $S_j(Q)=2^jQ\backslash 2^{j-1}Q$ for any integer $j\geq 2$. Then for all \(p,q\) satisfying \(p_0<p\leq q<q_0\), and for all \(f\in L^p(\mathbb R^n)\) supported in \(Q\), we have
		\begin{equation}\label{eq1-aim3}
			\begin{aligned}
				\Big(\fint_{S_j(Q)}|T_tf|^q\Big)^{1/q}&\lesssim \max\Big\{\Big(\f{r_Q}{t^{1/\alpha}}\Big)^{n/p_0}, \Big(\f{r_Q}{t^{1/\alpha}}\Big)^{n}\Big\}\Big(1+\f{t^{1/\alpha}}{2^jr_Q}\Big)^{n/q}\\
				&\qquad \cdot\Big(1+\f{2^jr_Q}{t^{1/\alpha}}\Big)^{-n-\epsilon} \frac{1}{|Q|^{1/p_0}} \|f\|_{L^{p_0,1}(Q)}.
			\end{aligned}
		\end{equation}
		
	\end{lemma}
	\begin{proof}
		%line 292
		Since the kernel of the operator \(T_{t}\), $t>0$ satisfies the estimate as in \eqref{eq:heat-kernel-bounds0}, we have
		\begin{align*}
			|T_t (f)(x)|&=\Big|\int_{\mathbb R^n} T_t(x,y) f(y)dy   \Big|\\
			&\leq \int_{\mathbb R^n} |T_t(x,y)| |f(y)| dy\\
			&\lesssim \int_{\mathbb{ R}^n} t^{-n/\alpha} \Big(\frac{t^{1/\alpha}+|x-y|}{t^{1/\alpha}}\Big)^{-n-\epsilon} \Big(1+\frac{t^{1/\alpha}}{|x|}\Big)^{\theta}\Big(1+\frac{t^{1/\alpha}}{|y|}\Big)^{\sigma} |f(y)| dy.
		\end{align*}
		Without loss of generality, we assume that $Q$ is centred at \(0\) with side length $r_Q$.
		We decompose \(\mathbb{R}^n \) into \(\tilde{Q}\) and \(\mathbb{R}^n\setminus \tilde{Q}\), where the cube \(\tilde{Q}\) is centred at \(0\) with side length \(t^{1/\alpha}\).
		Then we have
		\begin{align*}
			|T_t (f)(x)| &\leq \int_{\tilde{Q}} t^{-n/\alpha} \Big(\frac{t^{1/\alpha}+|x-y|}{t^{1/\alpha}}\Big)^{-n-\epsilon} \Big(1+\frac{t^{1/\alpha}}{|x|}\Big)^{\theta} \Big(1+\frac{t^{1/\alpha}}{|y|}\Big)^{\sigma} |f(y)| dy\\
			&\quad +  \int_{\mathbb{R}^n \setminus \tilde{Q}}t^{-n/\alpha} \Big(\frac{t^{1/\alpha}+|x-y|}{t^{1/\alpha}}\Big)^{-n-\epsilon} \Big(1+\frac{t^{1/\alpha}}{|x|}\Big)^{\theta} \Big(1+\frac{t^{1/\alpha}}{|y|}\Big)^{\sigma} |f(y)| dy.
		\end{align*}
		
		For \(p_0<p\leq q<q_0\), we further have
		\begin{align*}
			\|T_t (f)\|_{L^{q}(S_j(Q))}  \leq \text{Term} \mathrm{I_1}+\text{Term} \mathrm{I_2},
		\end{align*}
		where
		\begin{align*}
			\text{Term}
			\mathrm{I_1} := \Big\|\int_{\tilde{Q}} t^{-n/\alpha} \Big(1+\frac{|\cdot-y|}{t^{1/\alpha}}\Big)^{-n-\epsilon} \Big(1+\frac{t^{1/\alpha}}{|\cdot|}\Big)^{\theta} \Big(1+\frac{t^{1/\alpha}}{|y|}\Big)^{\sigma} |f(y)| dy\Big\|_{L^q(S_{j}(Q))}
		\end{align*}
		and
		\begin{align*}
			\text{Term} \mathrm{I_2}:= \Big\|\int_{\mathbb{R}^n \setminus \tilde{Q}}t^{-n/\alpha} \Big(1+\frac{|\cdot-y|}{t^{1/\alpha}}\Big)^{-n-\epsilon} \Big(1+\frac{t^{1/\alpha}}{|\cdot|}\Big)^{\theta} \Big(1+\frac{t^{1/\alpha}}{|y|}\Big)^{\sigma} |f(y)| dy\Big\|_{L^q(S_{j}(Q))}.
		\end{align*}
		
		For all \(f\in L^{p}(\mathbb{R}^n)\) supported in \(Q\), we consider the following two cases.
		
		\textbf{Case 1:} \(t^{1/\alpha}\leq r_{Q}\).
		Now we estimate \(\text{Term}\mathrm{I_1} \). For any $y \in \tilde{Q}$ and $x \in S_{j}(Q)$, since \(|y|\leq t^{1/\alpha} \), we get \(1+\frac{t^{1/\alpha}}{|y|} \approx \frac{t^{1/\alpha}}{|y|} \) and \(|x-y|\approx 2^{j}r_{Q}\) for \(j\ge 2\). We have
		\begin{equation*}%\label{eq: TermI_1}
			\begin{aligned}
				\text{Term} \mathrm{I_1}
				&\approx t^{-n/\alpha} \Big(1+\frac{2^{j}r_Q}{t^{1/\alpha}}\Big)^{-n-\epsilon}
				\Big( \int_{S_{j}(Q)} \Big| \int_{\tilde{Q}}   \Big(1+\frac{t^{1/\alpha}}{|x|}\Big)^{\theta} \Big(\frac{t^{1/\alpha}}{|y|}\Big)^{\sigma} |f(y)| dy
				\Big|^q dx \Big)^{1/q}\\
				&=t^{-n/\alpha} \Big(1+\frac{2^{j}r_{Q}}{t^{1/\alpha}}\Big)^{-n-\epsilon} \Big(\int_{S_{j}(Q)} \Big(1+\frac{t^{1/\alpha}}{|x|}\Big)^{\theta q} \Big(\int_{\tilde{Q}} \Big(\frac{t^{1/\alpha}}{|y|}\Big)^{\sigma} |f(y)| dy \Big)^{q} dx \Big)^{1/q}\\
				&\leq t^{-n/\alpha} \Big(1+\frac{2^{j}r_{Q}}{t^{1/\alpha}}\Big)^{-n-\epsilon} (r_{Q})^{\sigma} \Big(
				\int_{S_{j}(Q)} \Big(1+\frac{t^{1/\alpha}}{|x|}\Big)^{\theta q} \Big(\int_{\tilde{Q}}  \frac{|f(y)|}{|y|^{\sigma}}  dy \Big)^q dx
				\Big)^{1/q}.
			\end{aligned}
		\end{equation*}

		Recall that \(q<q_{0}= \frac{n}{\theta}\). Since \(x\in S_{j}(Q)=2^{j}Q\setminus 2^{j-1}Q\) for \(j\ge 2\), we have \(|x|\approx 2^{j}r_Q\) for \(j\ge 2\) and we further obtain
		\begin{equation}\label{eq: TermI_2}
			\begin{aligned}
				\text{Term} \mathrm{I_1}
				&  \leq \frac{(r_{Q})^{\sigma}}{t^{n/\alpha}} \Big(1+\frac{2^{j}r_{Q}}{t^{1/\alpha}}\Big)^{-n-\epsilon}
				\Big(
				\int_{S_{j}(Q)} \Big(1+\frac{t^{1/\alpha}}{|x|}\Big)^{n} \Big(\int_{\tilde{Q}}  \frac{|f(y)|}{|y|^{\sigma}}  dy \Big)^q dx
				\Big)^{1/q}\\
				&\lesssim \frac{(r_{Q})^{\sigma}}{t^{n/\alpha}}\Big(1+\frac{2^{j}r_{Q}}{t^{1/\alpha}}\Big)^{-n-\epsilon}  \Big(1+\frac{t^{1/\alpha}}{2^j r_Q}\Big)^{n/q} |S_{j}(Q)|^{1/q}
				\int_{\tilde{Q}}  \frac{|f(y)|}{|y|^{\sigma}}  dy.
			\end{aligned}
		\end{equation}
		
		The H\"older's inequality gives
		\begin{equation}\label{eq: TermI_3}
			\begin{aligned}
				\int_{\tilde{Q}}  \frac{|f(y)|}{|y|^{\sigma}}  dy \leq
				\|f\|_{L^{p_0,1}(Q)}\left\|   \frac{1}{|\cdot|^{\sigma}}\right\|_{L^{p_0',\infty}(\tilde{Q})},
			\end{aligned}
		\end{equation}
		where \(p_0=\frac{n}{n-\sigma}\) and \(p_0'=\frac{n}{\sigma}\).
		
		Let \(h(y):=\frac{1}{|y|^{\sigma}}\). By \cite[Proposition 1.4.5. (16)]{Gr1}, we have
		\begin{equation}\label{eq: TermI_4}
			\begin{aligned}
				\|h\|_{L^{p_0',\infty}(\tilde{Q})}
				%& = \sup_{\lambda>0} \{\lambda d_{h}^{1/p_{0}'}(\lambda) \}\\
				&=\sup_{\lambda>0} \Big\{ \lambda \Big|\Big\{y\in \tilde{Q}: \frac{1}{|y|^{\sigma}}> \lambda \Big\} \Big|^{1/p_0'}
				\Big\} \\
				&\leq \sup_{\lambda>0} \Big\{\lambda
				\Big(\min\Big\{ t^{1/\alpha},
				\ \frac{1}{\lambda^{1/\sigma}}\Big\}\Big)^{n/p_0'} \Big\}\\
				&\leq C.
			\end{aligned}
		\end{equation}
		Therefore, by combining %\eqref{eq: TermI_1},
		\eqref{eq: TermI_2}, \eqref{eq: TermI_3} and \eqref{eq: TermI_4}, we get
		\begin{align*}
			\text{Term}\mathrm{I_1} &\leq   t^{-n/\alpha} \Big(1+\frac{2^{j}r_{Q}}{t^{1/\alpha}}\Big)^{-n-\epsilon} (r_{Q})^{\sigma} \Big(1+\frac{t^{1/\alpha}}{2^j r_Q}\Big)^{n/q} |S_{j}(Q)|^{1/q}
			\int_{\tilde{Q}}  \frac{|f(y)|}{|y|^{\sigma}}  dy\\
			&\leq C  t^{-n/\alpha} \Big(1+\frac{2^{j}r_{Q}}{t^{1/\alpha}}\Big)^{-n-\epsilon} (r_{Q})^{\sigma} \Big(1+\frac{t^{1/\alpha}}{2^j r_Q}\Big)^{n/q} |S_{j}(Q)|^{1/q}  \|f\|_{L^{p_0,1}(Q)}\\
			&= C \Big( \frac{r_Q}{t^{1/\alpha}} \Big)^{n} \Big(1+\frac{2^{j}r_Q}{t^{1/\alpha}}\Big)^{-n-\epsilon} \Big(1+\frac{t^{1/\alpha}}{2^j r_{Q}}\Big)^{n/q}
			(r_{Q})^{\sigma-n}
			|S_{j}(Q)|^{1/q}   \|f\|_{L^{p_0,1}(Q)}.
		\end{align*}
		
		Since \(r_{Q}=|Q|^{1/n}\) and \(p_0=\frac{n}{n-\sigma}\), we obtain
		\begin{align*}
			\text{Term}\mathrm{I_1}	\lesssim \Big( \frac{r_Q}{t^{1/\alpha}} \Big)^{n} \Big(1+\frac{2^{j}r_Q}{t^{1/\alpha}}\Big)^{-n-\epsilon} \Big(1+\frac{t^{1/\alpha}}{2^j r_{Q}}\Big)^{n/q}
			|Q|^{1/p_0}
			|S_{j}(Q)|^{1/q}   \|f\|_{L^{p_0,1}(Q)}.
		\end{align*}

			We turn to estimate \(\text{Term}\mathrm{I_2}\).  Since \(f(y)\) is supported in \(Q\),  we have \(|y|\leq r_Q\). In the case \(y\in \mathbb{R}^n\setminus \tilde{Q}\) and \(x\in S_{j}(Q)\), we still have \(|x-y|\approx 2^{j} r_{Q}\) for \(j\ge 2\), then we further have
			\begin{equation*}
				%\label{eq: TermII_step1}
				\begin{aligned}
					&\quad\text{Term}\mathrm{I_2}\\ &\approx \Big\|
					\int_{\mathbb{R}^n\setminus \tilde{Q}} t^{-n/\alpha}  \Big(1+\frac{2^{j}r_Q}{t^{1/\alpha}}\Big)^{-n-\epsilon} \Big(1+\frac{t^{1/\alpha}}{|x|}\Big)^{\theta} \Big(1+\frac{t^{1/\alpha}}{|y|}\Big)^{\sigma} |f(y)| dy
					\Big\|_{L^q(S_j(Q))}\\
					&= t^{-n/\alpha} \Big(1+\frac{2^jr_Q}{t^{1/\alpha}}\Big)^{-n-\epsilon}
					\Big(\int_{S_j(Q)}
					\Big|	\int_{\mathbb{R}^n\setminus \tilde{Q}}
					\Big(1+\frac{t^{1/\alpha}}{|x|}\Big)^{\theta} \Big(1+\frac{t^{1/\alpha}}{|y|}\Big)^{\sigma} |f(y)| dy
					\Big|^{q} dx
					\Big)^{1/q}\\
					&\lesssim t^{-n/\alpha} \Big(1+\frac{2^jr_Q}{t^{1/\alpha}}\Big)^{-n-\epsilon} \Big(1+\frac{t^{1/\alpha}}{2^{j}r_Q}\Big)^{n/q} |S_{j}(Q)|^{1/q} \int_{\{y:\ t^{1/\alpha}<|y|\leq r_Q\}} \Big(1+\frac{t^{1/\alpha}}{|y|}\Big)^{\sigma} |f(y)| dy.
				\end{aligned}
			\end{equation*}
			
			Since \(t^{1/\alpha}<|y|\leq r_Q\), we get
			\begin{equation}\label{eq: estimate 1 for t and |y|}
				1+\frac{t^{1/\alpha}}{|y|}\leq \frac{r_Q}{|y|}+\frac{t^{1/\alpha}}{|y|} \lesssim \frac{r_Q}{|y|}.
			\end{equation}
			By inequality \eqref{eq: estimate 1 for t and |y|} and using H\"older's inequality again, we have
			\begin{align*}
				\int_{\{y:\ t^{1/\alpha}<|y|\leq r_Q\}} \Big(1+\frac{t^{1/\alpha}}{|y|}\Big)^{\sigma} |f(y)| dy
				& \lesssim (r_Q)^{\sigma}
				\int_{Q}\frac{1}{|y|^{\sigma}} |f(y)|dy\\
				&\leq (r_Q)^{\sigma} \|f\|_{L^{p_0,1}(Q)} \left\|\frac{1}{|\cdot|^{\sigma}} \right\|_{L^{p_0',\infty}(Q)}\\
				&\lesssim (r_Q)^{\sigma} \|f\|_{L^{p_0,1}(Q)}.
			\end{align*}
			Therefore, we obtain
			\begin{align*}\text{Term}\mathrm{I_2} \lesssim  \Big(\frac{r_Q}{t^{1/\alpha}}\Big)^{n} \Big(1+\frac{2^j r_Q}{t^{1/\alpha}}\Big)^{-n-\epsilon} \Big(1+\frac{t^{1/\alpha}}{2^jr_Q}\Big)^{n/q} |Q|^{-1/p_0} |S_j(Q)|^{1/q} \|f\|_{L^{p_0,1}(Q)}.
			\end{align*}

			\textbf{Case 2:} \(t^{1/\alpha}>r_Q\). We notice that in this case, \(\text{Term}\mathrm{I_2}=0\), since \(Q\subset \tilde{Q}\) and \(f\) is supported in \(Q\). It suffices to estimate \(\text{Term}\mathrm{I_1}\). Since \(0\leq |y|\leq r_Q<t^{1/\alpha}\), we have \(1+\frac{t^{1/\alpha}}{|y|} \approx \frac{t^{1/\alpha}}{|y|}\). Substituting the equivalence \(|x-y| \approx 2^{j} r_Q\) and \(|x|\approx 2^jr_Q\) again, we have
			\begin{align*}\text{Term}\mathrm{I_1} &\approx t^{-n/\alpha+\sigma/\alpha}
				\Big(1+\frac{2^j r_Q}{t^{1/\alpha}}\Big)^{-n-\epsilon}
				\Big(\int_{S_{j}(Q)} \Big|\int_{Q} \Big(1+\frac{t^{1/\alpha}}{|x|}\Big)^{\theta} \frac{1}{|y|^{\sigma}} |f(y)| dy \Big|^q dx \Big)^{1/q}\\
				&\lesssim t^{(\sigma-n)/\alpha}
				\Big(1+\frac{2^jr_Q}{t^{1/\alpha}}\Big)^{-n-\epsilon}
				\Big(1+\frac{t^{1/\alpha}}{2^jr_Q}\Big)^{\theta q_0/q} |S_j(Q)|^{1/q} \int_{Q} \frac{1}{|y|^{\sigma}} |f(y)| dy\\
				&\leq t^{-n/(p_0\alpha)}
				\Big(1+\frac{2^jr_Q}{t^{1/\alpha}}\Big)^{-n-\epsilon} \Big(1+\frac{t^{1/\alpha}}{2^jr_Q}\Big)^{n/q}   |S_j(Q)|^{1/q} \int_{Q} \frac{1}{|y|^{\sigma}} |f(y)|dy.
			\end{align*}
			
			Since the H\"older's inequality gives
			\begin{align*}
				\int_{Q} \frac{1}{|y|^{\sigma}} |f(y)|dy \leq \|f\|_{L^{p_0,1}(Q)} \left\|\frac{1}{|\cdot|^{\sigma}}\right\|_{L^{p_0',\infty}(Q)}\lesssim \|f\|_{L^{p_0,1}(Q)},
			\end{align*}
			we obtain
			\begin{align*}
				\text{Term}\mathrm{I_1}
				&\lesssim t^{-n/(p_0\alpha)}
				\Big(1+\frac{2^jr_Q}{t^{1/\alpha}}\Big)^{-n-\epsilon} \Big(1+\frac{t^{1/\alpha}}{2^jr_Q}\Big)^{n/q}   |S_j(Q)|^{1/q}
				\|f\|_{L^{p_0,1}(Q)}\\
				&= \Big(\frac{r_Q}{t^{1/\alpha}}\Big)^{n/p_0}
				\Big(1+\frac{2^jr_Q}{t^{1/\alpha}}\Big)^{-n-\epsilon} \Big(1+\frac{t^{1/\alpha}}{2^jr_Q}\Big)^{n/q}   |S_j(Q)|^{1/q} \frac{1}{|Q|^{1/p_0}}
				\|f\|_{L^{p_0,1}(Q)},
			\end{align*}
			where the last equality is from \(\frac{1}{|Q|^{1/p_0}}=r_{Q}^{-n/p_0}\).
			Then, by the above estimates for \(\text{Term}\mathrm{I_1}\) and \(\text{Term}\mathrm{I_2}\), together with
			\begin{align*}
				\Big(\fint_{S_j(Q)}|T_tf|^q\Big)^{1/q} & = |S_j(Q)|^{-1/q} \|T_t (f)\|_{L^{q}(S_j(Q))}\\
				& \leq |S_j(Q)|^{-1/q} (\text{Term}\mathrm{I_1}+\text{Term}\mathrm{I_2}),
			\end{align*}
			the inequality \eqref{eq1-aim3} immediately follows.
		\end{proof}

		Denote by $T^*_t$ the adjoint operator of $T_t$. Then we have:
		\begin{lemma}\label{lem: adjoint_operator_analogue}
			Give a cube $Q$ with the side length $r_Q$. Denote $S_j(Q)=2^jQ\backslash 2^{j-1}Q$ for $j\geq 2$. For all \(q_0'<q'\leq p'<p_0'\), and for all \(f\in L^{q'}(S_j(Q))\) supported in \(S_j(Q)\),
			\begin{equation}\label{eq1-lem2.3}
				\begin{aligned}
					\frac{1}{|Q|^{1/p_0'}} \|T_t^{*}f\|_{L^{p_0',\infty }(Q)}
					& \lesssim \max\Big\{\Big(\f{2^jr_Q}{t^{1/\alpha}}\Big)^n,\Big(\f{2^jr_Q}{t^{1/\alpha}}\Big)^{n/q'} \Big\} \Big(1+\f{t^{1/\alpha}}{r_Q}\Big)^{n/p_0'}\\
					&\qquad \cdot\Big(1+\f{2^jr_Q}{t^{1/\alpha}}\Big)^{-n-\epsilon} \Big(\fint_{S_j(Q)}|f|^{q'}\Big)^{1/{q'}}.
				\end{aligned}
			\end{equation}
			In addition, for $j=0,1$ and $t \approx r_Q^\alpha$ we have
			\begin{equation}\label{eq2-lem2.3}
				\frac{1}{|Q|^{1/p_0'}} \|T_t^{*}f\|_{L^{p_0',\infty }(Q)}
				\lesssim  \Big(\fint_{S_j(Q)}|f|^{q'}\Big)^{1/{q'}}.
			\end{equation}
		\end{lemma}
		
		\begin{proof}
			We need only to give the proof for \eqref{eq1-lem2.3} since the proof of \eqref{eq2-lem2.3} is similar.
			
			To prove \eqref{eq1-lem2.3} , by duality it suffices to prove  for all $f\in L^p(S_j(Q))$ supported in \(S_{j} (Q)\) with $j\ge 2$, we have
			\begin{equation*}%\label{eq2-aim4}
				\begin{aligned}
					\frac{1}{|Q|^{1/q_0}} \|T_tf\|_{L^{q_0,\infty }(Q)}
					& \lesssim \max\Big\{\Big(\f{2^jr_Q}{t^{1/\alpha}}\Big)^n,\Big(\f{2^jr_Q}{t^{1/\alpha}}\Big)^{n/p} \Big\} \Big(1+\f{t^{1/\alpha}}{r_Q}\Big)^{n/q_0}\\
					&\qquad \cdot\Big(1+\f{2^jr_Q}{t^{1/\alpha}}\Big)^{-n-\epsilon} \Big(\fint_{S_j(Q)}|f|^p\Big)^{1/p}.
				\end{aligned}
			\end{equation*}

			Since \(f\) is supported in \(S_j(Q)\), we have
			\begin{equation*}
				\begin{aligned}
					|T_t(f)(x)|&=\Big|\int_{\mathbb R^n} T_{t}(x,y)f(y)dy\Big|\\
					&\lesssim t^{-n/\alpha} \Big(1+\frac{t^{1/\alpha}}{|x|}\Big)^{\theta}
					\int_{S_{j}(Q)}  \Big(\frac{t^{1/\alpha}+|x-y|}{t^{1/\alpha}}\Big)^{-n-\epsilon} \Big(1+\frac{t^{1/\alpha}}{|y|}\Big)^{\sigma} |f(y)| dy.
				\end{aligned}
			\end{equation*}
			Without loss of generality, we assume that the cube \(Q\) is centred at origin with side length \(r_Q\). For any \(x\in Q\), it is obvious from \(k\geq 2\) that \(|x-y|\approx 2^jr_Q\). We further have
			\begin{align*}\label{eq: Lemma_T_t}
				|T_t(f)(x)|
				&\lesssim  \Big(1+\frac{2^jr_Q}{t^{1/\alpha}}\Big)^{-n-\epsilon}
				t^{-n/\alpha}\Big(1+\frac{t^{1/\alpha}}{|x|} \Big)^{\theta}
				\int_{S_j(Q)}  \Big(1+\frac{t^{1/\alpha}}{|y|}\Big)^{\sigma}
				|f(y)|dy.
			\end{align*}
			Denote
			\[\text{Term}\mathrm{II}:=t^{-n/\alpha}\Big(1+\frac{t^{1/\alpha}}{|x|} \Big)^{\theta}
			\int_{S_j(Q)}  \Big(1+\frac{t^{1/\alpha}}{|y|}\Big)^{\sigma}
			|f(y)|dy.\]
			Then,
			\begin{align}\label{eq: Lemma_T_t}
				\frac{1}{|Q|^{1/q_0}} \|T_tf\|_{L^{q_0,\infty }(Q)} \lesssim
				\frac{1}{|Q|^{1/q_0}}       \Big(1+\frac{2^jr_Q}{t^{1/\alpha}}\Big)^{-n-\epsilon}   \|\text{Term}\mathrm{II}\|_{L^{q_0,\infty }(Q)}.
			\end{align}
			We will estimate \(\|\text{Term}\mathrm{II}\|_{L^{q_0,\infty }(Q)}\)          in two cases:  \(t^{1/\alpha}\leq 2^jr_Q\) and  \(t^{1/\alpha}> 2^jr_Q\).
			
			\textbf{Case 1:} \(t^{1/\alpha}\leq 2^jr_Q\). For any \(x\in Q\) and \(y \in S_j(Q)\), we have \(|x|<r_Q\) and \(2^{j-1} r_Q\leq |y| < 2^j r_Q\). Then, we have
			\begin{align*}
				\text{Term}\mathrm{II}
				&= t^{-n/\alpha} (|x|+t^{1/\alpha})^{\theta}
				\frac{1}{|x|^{\theta} }
				\int_{S_j(Q)} (|y|+t^{1/\alpha})^{\sigma}
				\frac{1}{|y|^{\sigma} }
				|f(y)|dy\\
				&\leq
				t^{-n/\alpha} (r_Q+t^{1/\alpha})^{\theta}
				\frac{1}{|x|^{\theta} }
				\int_{S_j(Q)} (2^{j+1}r_Q)^{\sigma}
				\frac{1}{(2^{j-1}r_Q)^{\sigma} }
				|f(y)|dy\\
				&\lesssim
				t^{-n/\alpha}    (r_Q+t^{1/\alpha})^{\theta}
				\frac{1}{|x|^{\theta} }
				\int_{S_j(Q)}
				|f(y)|dy.
			\end{align*}
			Using H\"{o}lder's inequality for \(p\in (p_0,q]\),  we have
			\begin{equation}\label{intfonSjQ}
				\begin{aligned}
					\int_{S_j(Q)} |f(y)|dy
					&\leq \|f\|_{L^{p}({S_{j}(Q))}}
					|S_{j}(Q)|^{1/p'}\\
					%	&=  \|f\|_{L^{p}({S_{j}(Q))}}
					%	(2^jr_Q) ^{-\sigma p'}
					%	|S_{j}(Q)|^{1-\frac{1}{p}}\\
					&\lesssim  \|f\|_{L^{p}({S_{j}(Q))}}  (2^jr_Q) ^{n}  |S_j(Q)|^{-1/p}.
				\end{aligned}
			\end{equation}
			Then we get
			\begin{align*}
				\text{Term}\mathrm{II}
				\lesssim
				t^{-n/\alpha}
				(2^jr_Q) ^{n}(r_Q+t^{1/\alpha})^{\theta}
				\frac{1}{|x|^{\theta} }
				\|f\|_{L^{p}({S_{j}(Q))}}    |S_j(Q)|^{-1/p}.
			\end{align*}
			Taking the \(\|\cdot\|_{L^{q_0,\infty}(Q)}\) norm of both sides, we have
			\begin{align*}
				\| \text{Term}\mathrm{II}\|_{L^{q_0,\infty}(Q)} &\lesssim t^{-n/\alpha} (2^jr_Q)^n (r_Q+t^{1/\alpha})^\theta  |S_j(Q)| ^{-1/p}\|f\|_{L^{p}(S_j(Q))}   \bigg\|\frac{1}{|\cdot|^\theta}
				\bigg\|_{L^{q_0,\infty}(Q)}\\
				&\lesssim  t^{-n/\alpha} (2^jr_Q)^n (r_Q+t^{1/\alpha})^\theta  |S_j(Q)| ^{-1/p}\|f\|_{L^{p}(S_j(Q))}.
			\end{align*}
			That is,
			\begin{align*}
				\| \text{Term}\mathrm{II}\|_{L^{q_0,\infty}(Q)}  &\lesssim
				t^{-n/\alpha} (2^jr_Q)^n \Big(1+\frac{t^{1/\alpha}}{r_Q}\Big)^{n/q_0}  r_{Q}^{n/q_0} |S_j(Q)| ^{-1/p}\|f\|_{L^{p}(S_j(Q))}\\
				&= \Big(\frac{2^jr_Q}{t^{1/\alpha}}\Big)^{n} \Big(1+\frac{t^{1/\alpha}}{r_Q}\Big)^{n/q_0}
				|Q|^{1/q_0} |S_j(Q)|^{-1/p} \|f\|_{L^p(S_j(Q))}.
			\end{align*}
			
			\textbf{Case 2:} \(t^{1/\alpha}>2^jr_Q.\) For any \(y\in S_j(Q)\), we have \(2^{j-1}r_Q<|y|\leq 2^{j}r_Q<t^{1/\alpha}\), which implies
			\begin{equation}\label{eq: |y| and t}
				1+\frac{t^{1/\alpha}}{|y|}\lesssim \frac{t^{1/\alpha}}{|y|}\lesssim \frac{t^{1/\alpha}}{2^{j}r_Q}.
			\end{equation}
			For \(x\in Q\), by applying
			\eqref{intfonSjQ} and
			\eqref{eq: |y| and t}, we have
			\begin{align*}
				\text{Term}\mathrm{II}
				&\lesssim
				\Big(1+\frac{t^{1/\alpha}}{|x|}\Big)^{\theta}
				\Big(\frac{2^jr_Q}{t^{1/\alpha}}\Big)^{n-\sigma}
				\|f\|_{L^p(S_j(Q))}  |S_j(Q)|^{-1/p}\\
				&= (|x|+t^{1/\alpha})^{\theta} \frac{1}{|x|^{\theta}} \Big(\frac{2^jr_Q}{t^{1/\alpha}}\Big)^{n/p_0}  \|f\|_{L^p(S_j(Q))}  |S_j(Q)|^{-1/p}\\
				&\leq
				\Big(1+\frac{t^{1/\alpha}}{r_Q}\Big)^{n/q_0}
				\Big(\frac{r_{Q}}{|x|}\Big)^{{n/q_0}} \Big(\frac{2^jr_Q}{t^{1/\alpha}}\Big)^{n/p_0}  \|f\|_{L^p(S_j(Q))}  |S_j(Q)|^{-1/p},
			\end{align*}
			where the equality is from \(p_0=\frac{n}{n-\sigma}\), and the last inequality follows from \(|x|<r_Q\) and \(\theta=\frac{n}{q_0}\).
			
			Taking \(\|\cdot\|_{L^{q_0,\infty}(Q)}\) norm on both sides, by \(p_0<p\) we get\begin{equation*}
				\begin{aligned}
					\| \text{Term}\mathrm{II}\|_{L^{q_0,\infty}(Q)} & \lesssim \Big(\frac{2^jr_Q}{t^{1/\alpha}}\Big)^{n/p_0} \Big(1+\frac{t^{1/\alpha}}{r_Q}\Big)^{n/q_0} |Q|^{1/q_0} |S_j(Q)|^{-1/p} \|f\|_{L^p(S_j(Q))}\\
					& \leq\Big(\frac{2^jr_Q}{t^{1/\alpha}}\Big)^{n/p} \Big(1+\frac{t^{1/\alpha}}{r_Q}\Big)^{n/q_0} |Q|^{1/q_0} |S_j(Q)|^{-1/p} \|f\|_{L^p(S_j(Q))}.
				\end{aligned}
			\end{equation*}
			Therefore, by \eqref{eq: Lemma_T_t}, we obtain
			\begin{align*}
				\frac{1}{|Q|^{1/q_0}} \|T_tf\|_{L^{q_0,\infty }(Q)}
				& \lesssim \max\Big\{\Big(\f{2^jr_Q}{t^{1/\alpha}}\Big)^n,\Big(\f{2^jr_Q}{t^{1/\alpha}}\Big)^{n/p} \Big\} \Big(1+\f{t^{1/\alpha}}{r_Q}\Big)^{n/q_0}\\
				&\qquad \cdot\Big(1+\f{2^jr_Q}{t^{1/\alpha}}\Big)^{-n-\epsilon} \Big(\fint_{S_j(Q)}|f|^p\Big)^{1/p}.
			\end{align*}
			This completes the proof of the lemma.
		\end{proof}

		\section{Endpoint estimates of square operators}\label{section: square operators}

		In this section, we aim to prove Theorem \ref{thm: square function}.  Note that, under the setting \(1\leq p_{0}:= \frac{n}{n-\sigma}  < 2 <\frac{n}{\theta}=:q_0\), the space \(L^{p_0,\infty}(\mathbb R^n)\) is normable (see, \cite[Exercise 1.1.12(c)]{Gr1}), i.e., there is a norm on the space equivalent to the quasi-norm \(\|\cdot\|_{L^{p_0,\infty}(\mathbb R^n)}\). Therefore, to prove that the square operator \(S_L\) is bounded from \(L^{p_0,1}(\mathbb R^n)\) to \(L^{p_0,\infty}(\mathbb R^n)\), it suffices to prove that \(S_{L}\) is of restricted weak type \( (p_{0}, p_0) \) (see, e.g., \cite[Theorem 3.13]{S4}); that is,  for all \( \lambda>0\), and for any measurable set \(E\) of finite measure,
		\begin{equation}\label{eq:aim1}
			|\{ x\in \mathbb R^{n}: |S_{L}\chi_{E}(x)|>\lambda \}| \lesssim \lambda^{-p_0}\|\chi_{E}\|_{L^{p_0}(\mathbb R^{n})}^{p_{0}}.
		\end{equation}

		Make a Calder\'on--Zygmund decomposition of \(\chi_{E}\) in \(L^{p_0}(\mathbb R^n)\) at height \(\lambda>0\).
		We get a sequence of maximal disjoint dyadic cubes \(\{Q_j\}_j\) satisfying that
		\begin{enumerate}[label=(cz-\roman*)]
			\item \(\chi_{E } =g+b=g+\sum_{j}b_j\);
			\item\label{it-2} \(|g(x)| \leq C\lambda\) for almost \(x\in \mathbb R^n\) and \(\|g\|_{L^{p_0}(\mathbb R^n)}\leq \|\chi_{E}\|_{L^{p_0}(\mathbb R^n)}\);
			\item
			the support of each \(b_{j}\) is contained in \(Q_{j}\),
			\[
			\int_{\mathbb R^n} b_j(x)dx=0
			\quad \text{and} \quad
			\int_{\mathbb R^n} |b_j(x)|^{p_0} dx\leq C\lambda^{p_0} |Q_{j}|;
			\]
			\item \(
			|Q_j|\leq \lambda^{-p_0} \int_{Q_j} \chi_{E}^{p_0} (x) dx\leq 2^n |Q_j|\);
			\item \(\sum_{j}|Q_{j}| \leq C\lambda^{-p_0} \int_{\mathbb R^n}|\chi_{E}(x)|^{p_0}dx \).
		\end{enumerate}

		From the Calder\'on--Zygmund decomposition of \(\chi_{E}\) we have
		\begin{align*}
			|\{ x\in \mathbb R^{n}: |S_{L}\chi_{E}(x)|>\lambda \}|  &\leq
			\Big|\Big\{x\in \mathbb R^n: |S_{L}g(x)>\frac{\lambda}{2} \Big\}\Big|+\Big|\Big\{x\in \mathbb R^n: |S_{L}b(x)|>\frac{\lambda}{2} \Big\}\Big|.
			\end{align*}

		Using Chebyshev's inequality and the \(L^2\)-boundedness of \(S_L\), we get
		\begin{equation*}
			\Big|\Big\{x\in \mathbb R^n: |S_{L} g(x)|>\frac{\lambda}{2} \Big\}\Big| \lesssim  \lambda^{-2}\|S_Lg\|_{L^2}^2\lesssim \lambda^{-2}\|g\|_{L^2}^2\lesssim \lambda^{-p_0}\|g\|_{L^{p_0}}^{p_0}
        \lesssim \lambda^{-p_0}\|\chi_{E}\|_{L^{p_0}(\mathbb R^{n})}^{p_{0}},
		\end{equation*}
        where the third inequality follows from \ref{it-2}.

		We now focus on establishing the following estimate for the bad part,
		\begin{align}\label{eq:S_Lb}
			\Big|\Big\{x\in \mathbb R^n: |S_{L}b(x)|>\frac{\lambda}{2}\Big \}\Big| \lesssim \lambda^{-p_0}\|\chi_{E}\|_{L^{p_0}(\mathbb R^{n})}^{p_{0}}.
		\end{align}

		For each \(j\),
		\(b_{j}\) is supported in \(Q_{j}\). Denote by \(r_{j}\) the side length of \(Q_{j}\). Denote
		\[
		h_{1} = \sum_{j} (I-(I-e^{-r_{j}^{\alpha}L})^m) b_j \quad
		\text{and}
		\quad
		h_{2} =\sum_{j} (I-e^{-r_{j}^{\alpha}L})^m b_j,
		\]
		then we have
		\begin{equation*}
			b=\sum_{j}b_{j}=h_1+h_2.
		\end{equation*}
		
		Therefore, we get
		\begin{align*}
			\Big|\Big\{x\in \mathbb R^n: |S_{L}b(x)|>\frac{\lambda}{2} \Big\}\Big|
			\leq 	\Big|\Big\{x\in \mathbb R^n: |S_{L}h_1(x)|>\frac{\lambda}{4} \Big\}\Big|+	\Big|\Big\{x\in \mathbb R^n: |S_{L}h_2(x)|>\frac{\lambda}{4} \Big\}\Big|.
		\end{align*}

		\textbf{Proof of Term Involving \(\boldsymbol{h_1}\).}
		For convenience, we estimate the term \(|\{x\in \mathbb R^n: |S_{L}h_1(x)|>\lambda \} | \). The following claim is crucial for our proof.
		\begin{equation}\label{eq:h1 result}
			\|h_1\|_{L^2(\mathbb R^n)} = \|\sum_{j} (I-(I-e^{-r_j^{\alpha}L})^m)b_j\|_{L^2} \leq C\lambda \|\sum_{j}\chi_{Q_{j}}\|_{L^2}.
		\end{equation}
		Once we establish \eqref{eq:h1 result}, by the Chebyshev's theorem, Calder\'on--Zygmund decomposition and the \(L^2\) boundedness of \(S_{L}\), we can obtain the desired result,
		\begin{align*}
			|\{ x:|S_{L}(h_1) (x)|>\lambda\}| &\leq \lambda^{-2} \int_{\mathbb R^n} |S_{L}(h_1)(x)|^2dx\\
			&\leq C\lambda^{-2} \int_{\mathbb R^n}|h_1(x)|^2dx\\
			&\lesssim \lambda^{-p_0} \|\chi_{E}\|_{L^{p_0}}^{p_0}.
		\end{align*}
		
		In what follows, we prove \eqref{eq:h1 result}. By duality it is well-known that
		\begin{align*}
			\Big  \|\sum_{j} (I-(I-e^{-r_j^{\alpha}L})^m)b_j \Big\|_{L^2} &= \sup_{\|\phi\|_{L^2}\leq1} \Big|\Big \langle \sum_{j} (I-(I-e^{-r_j^{\alpha}L})^m)b_j , \phi  \Big\rangle\Big|\\
			&\leq \sup_{\|\phi\|_{L^2}\leq1}  \sum_{j} |\langle (I-(I-e^{-r_j^{\alpha}L})^m)b_j , \phi \rangle|,
		\end{align*}
		To prove the claim \eqref{eq:h1 result}, it is sufficient to prove
		\begin{align*}
			\sup_{\|\phi\|_{L^2}\leq1}  \sum_{j} \left|\langle (I-(I-e^{-r_j^{\alpha}L})^m)b_j , \phi \rangle \right|\leq C \lambda
			\|\sum_{j}\chi_{Q_{j}}\|_{L^2}.
		\end{align*}

		For \(k \in \mathbb{Z}^{+}\), let \(S_k(Q_j):=2^kQ_j\backslash 2^{k-1}Q_j\)  and define \(S_0(Q_j):=Q_j\).
		For any \(\phi\in L^{2}(\mathbb R^n)\), we split
		\[\phi=\sum_{k \in \mathbb{Z}^{+}\cup\{0\}} g_k, \text{ ~~where~~ } g_k=\phi \chi_{S_{k}(Q_j)}.\]
		Recall that \(p_{0}=\frac{n}{n-\sigma}\) and \(p_0'=\frac{n}{\sigma}\).
		The adjoint operator of \(I-(I-e^{-r_j^{\alpha}L})^m\), denoted by \( (I-(I-e^{-r_j^{\alpha}L})^m)^* \),  satisfies the following estimate
		\begin{align*}
			|\langle(I-(I-e^{-r_j^{\alpha}L})^m)^*\phi,b_{j}\rangle|
			& =|\langle (I-(I-e^{-r_j^{\alpha}L})^m)^* (\sum_k g_k), b_j  \rangle| \\
			&\leq \sum_{k \in \mathbb{Z}^{+}\cup\{0\}} \int_{\mathbb R^n} \chi_{Q_j} (I-(I-e^{-r_{j}^{\alpha}L})^m)^{*} g_k(x) \overline{b_j(x)}dx\\
			&\leq \sum_{k \in \mathbb{Z}^{+}\cup\{0\}}  \|b_j\|_{L^{p_0,1}}
			\| (I-(I-e^{-r_j^{\alpha}L})^m)^*g_k \|_{L^{{p_0'},\infty}(Q_j)},
		\end{align*}
		which implies
		\begin{align}\label{eq: inner_product}
			|\langle \phi, (I-(I-e^{-r_j^{\alpha}L})^m)b_{j}\rangle|
			\leq \sum_{k \in \mathbb{Z}^{+}\cup\{0\}}  \|b_j\|_{L^{p_0,1}}
			\| (I-(I-e^{-r_j^{\alpha}L})^m)^*g_k \|_{L^{{p_0'},\infty}(Q_j)}.
		\end{align}

		For \(k\geq 2\), applying Lemma \ref{lem: adjoint_operator_analogue} to
		the operator \[T_{r_j^{\alpha} }= I-(I-e^{-r_j^{\alpha}L})^m \]
		and the function \(g_k\),
		we obtain, for \(\frac{n}{n-\theta}=q'_0<r<2\)
			\begin{align*}
				&\quad \frac{1}{|Q_j|^{1/p_0'}}
				\| (I-(I-e^{-r_j^{\alpha}L})^m )^* g_k  \|_{L^{p_0',\infty}(Q_j)}\\ &  \leq C \max\{2^{kn},2^{kn/r} \} 2^{n/p_0'}2^{n/r}(1+2^k)^{-n-\alpha} \inf_{x'\in 2^kQ_j}M_{r}(g_k)(x') \\
				&\leq C 2^{kn} 2^{n/p_0'}2^{n/r} 2^{-k(n+\alpha)}  \inf_{x'\in 2^kQ_j}M_{r}(g_k)(x')\\
				&=C_{n,p_0'}  2^{-k\alpha} \inf_{x'\in 2^kQ_j}M_{r}(g_k)(x'),
			\end{align*}
			which implies
			\begin{equation}\label{eq: adjoint operator estimate}
				\begin{aligned}
					\| (I-(I-e^{-r_j^{\alpha}L})^m )^* g_k  \|_{L^{p_0',\infty}(Q_j)} \leq  C_{n,p_0'} |Q_j|^{1/p_0'}   2^{-k\alpha} \inf_{x'\in 2^kQ_j}M_{r}(g_k)(x').
				\end{aligned}
			\end{equation}

			By \eqref{eq: adjoint operator estimate} and \(Q_j\subset2^kQ_j\) for \(k\geq 2\), we get
			\begin{align*}
				&\qquad \sum_{k\geq 2}  \|b_j\|_{L^{p_0,1}}
				\| (I-(I-e^{-r_j^{\alpha}L})^m)^*g_k \|_{L^{{p_0'},\infty}(Q_j)}\\
				&\leq C_{n,p_0'} \|b_j\|_{L^{p_0,1}}   \sum_{k\geq 2}  |Q_j|^{1/p_0'}  2^{-k\alpha} \inf_{x'\in 2^kQ_j}M_{r}(g_k)(x')\\
				&\leq C_{n,p_0'}  \|b_j\|_{L^{p_0,1}}   \sum_{k\geq2}  |Q_j|^{1/p_0'}  2^{-k\alpha} \inf_{x'\in 2^kQ_j}M_{r}(\phi)(x')\\
				&\leq C_{n,p_0'} \lambda |Q_j| \sum_{k\geq 2}    2^{-k\alpha}
				\inf_{x'\in Q_j}M_{r}(\phi)(x'),
			\end{align*}	
			where \(\|b_j\|_{L^{p_0,1}(\mathbb R^n)}=\|b_j\|_{L^{p_0}(\mathbb R^n)} \leq C_n \lambda|Q_j|^{1/p_0}\) is from (cz-iii). We further have
			\begin{align*}
				\sum_{k\geq 2}  \|b_j\|_{L^{p_0,1}}
				\| (I-(I-e^{-r_j^{\alpha}L})^m)^*g_k \|_{L^{{p_0'},\infty}(Q_j)}
				&\leq C_{n,p_0'} \lambda \sum_{k\geq2}    2^{-k\alpha}  \int_{Q_j}\inf_{x'\in Q_j}M_{r}(\phi)(x')dx\\
				&\leq C_{n,p_0'}  \lambda \sum_{k\geq 2}    2^{-k\alpha}   \int_{Q_j}M_{r}(\phi)(x)dx\\
				&\lesssim
				\lambda \int_{Q_j}M_{r}(\phi)(x)dx.
			\end{align*}
			
			For \(k=1\), by Lemma \ref{lem: adjoint_operator_analogue}, we have
			\begin{align*}
				\frac{1}{|Q_j|^{1/p_0'}}
				\| (I-(I-e^{-r_j^{\alpha}L})^m )^* g_1  \|_{L^{p_0',\infty}(Q_j)}  \lesssim
				\inf_{x'\in 2Q_j}M_{r}(\phi)(x').
			\end{align*}
			It follows from \(Q_j \subset 2Q_j\) that
			\begin{align*}
				%|\langle \phi, (I-(I-e^{-r_j^{\alpha}L})^m)b_{j}\rangle|
				\|b_j\|_{L^{p_0,1}}
				\| (I-(I-e^{-r_j^{\alpha}L})^m)^*g_1 \|_{L^{{p_0'},\infty}(Q_j)}
				&\lesssim
				\lambda \int_{Q_j}\inf_{x'\in Q_j}M_{r}(\phi)(x')dx\\
				&\lesssim
				\lambda \int_{Q_j}M_{r}(\phi)(x)dx.
			\end{align*}
			
			For \(k=0\), by Lemma \ref{lem: adjoint_operator_analogue} again, we obtain
			\begin{align*}
				\frac{1}{|Q_j|^{1/p_0'}}
				\| (I-(I-e^{-r_j^{\alpha}L})^m )^* g_k  \|_{L^{p_0',\infty}(Q_j)}  & \lesssim
				\inf_{x'\in Q_j}M_{r}(\phi)(x').
			\end{align*}
			Then,  we have
			\begin{align*}
				\|b_j\|_{L^{p_0,1}}
				\| (I-(I-e^{-r_j^{\alpha}L})^m)^*g_0 \|_{L^{{p_0'},\infty}(Q_j)}
				&\lesssim  \lambda
				\int_{Q_j} \inf_{x'\in Q_j}M_{r}(\phi)(x')dx\\
				&\lesssim \lambda \int_{Q_j}M_{r}(\phi)(x)dx.
			\end{align*}
			
			Thus, from \eqref{eq: inner_product}, we have
			\begin{align*}
				\sup_{\|\phi\|_{L^{2}}\leq1} \sum_{j} |\langle (I-(I-e^{-r_j^{\alpha}L})^m)b_j, \phi \rangle|
				&\lesssim  \sup_{\| \phi\|_{L^2}\leq 1} \lambda\sum_{j} \int_{Q_j} M_{r}(\phi) (x) dx\\
				&=  \sup_{\| \phi\|_{L^2}\leq 1} \lambda \int_{\mathbb R^n} M_{r}(\phi) (x) \sum_{j} \chi_{Q_j}(x) dx.
			\end{align*}
			Since \(M_{r}\) is bounded on \(L^2(\mathbb R^n)\) for \(r<2\), using the H\"older's inequality we obtain
			\begin{align*}
				\sup_{\| \phi\|_{L^2}\leq 1} \lambda \int_{\mathbb R^n} M_{r}(\phi) (x) \sum_{j} \chi_{Q_j}(x) dx
				&\leq  \sup_{\| \phi\|_{L^2}\leq 1}  \lambda \|M_{r}(\phi)\|_{L^2(\mathbb R^n)} \big\|\sum_{j} \chi_{Q_j} \big\|_{L^2(\mathbb R^n)}\\
				&\lesssim  \sup_{\| \phi\|_{L^2}\leq 1}  \lambda \|\phi\|_{L^2(\mathbb R^n)} \big \|\sum_{j} \chi_{Q_j} \big\|_{L^2(\mathbb R^n)}\\
				&\lesssim \lambda  \big \|\sum_{j} \chi_{Q_j} \big\|_{L^2(\mathbb R^n)}.
			\end{align*}

			Therefore, the claim \eqref{eq:h1 result} is proved.
			
			\textbf{Proof of Term Involving \(\boldsymbol{h_2}\).}
			Now we turn to prove \[
			|\{ x:|S_{L}(h_2) (x)|>\lambda\}| \lesssim \lambda^{-p_0} \|\chi_E\|_{L^{p_0}}^{p_0} .\] Recall that
			\(
			h_{2} =\sum_{j} (I-e^{-r_{j}^{\alpha}L})^m b_j
			\).
By Chebyshev's inequality we get
			\begin{equation} \label{eq: S_L h_2 estimate}
				\begin{aligned}
					&\qquad \Big|
					\Big\{x\in \mathbb R^n: \big|
					S_L\big(\sum_{j}(I-e^{-r_j^\alpha L})^m b_j\big)(x)
					\big| >\lambda
					\Big\}
					\Big| \\
					&  \lesssim \sum_{j} |2Q_j| + \lambda^{-2} \Big\| S_L\big(
					\sum_{j}(I-e^{-r_j^\alpha L})^m b_j \big)
					\Big\|_{L^2(\mathbb R^n\setminus 2Q_j)}^2\\
					&  \lesssim \lambda^{-p_0} \|\chi_E\|_{L^{p_0}(\mathbb{R}^n)}^{p_0} + \lambda^{-2} \int_{\mathbb R^n\setminus 2Q_j}
					\Big| S_L\big(
					\sum_{j}(I-e^{-r_j^\alpha L})^m b_j \big)(x)
					\Big|^2 dx.
				\end{aligned}
			\end{equation}

			We first claim that
			\begin{align}\label{eq: sum of S_L}
				\|S_{L}(\sum_{j}(I-e^{-r_{j}^{\alpha} L} )^{m} b_j) \|_{L^2(\mathbb {R}^{n}\setminus 2Q_{j})} \leq \sum_{j} 	\|S_{L}((I-e^{-r_{j}^{\alpha} L} )^{m} b_j) \|_{L^2(\mathbb {R}^{n}\setminus 2Q_{j})}.
			\end{align}
			Indeed, let \(\textbf{u}=(u_1,u_2,\cdots,u_m)\) with \(0<u_1,u_2,\cdots,u_m<r_{j}^{\alpha}\) and  \(|\textbf{u}|=u_1+u_2+\cdots+u_m\), then we have the estimate
			\begin{equation}\label{eq: Proof of sum of S_L 1}
				\begin{aligned}
					&\qquad \Big\|S_{L}\Big(\sum_{j}(I-e^{-r_{j}^{\alpha} L} )^{m} b_j\Big) \Big\|_{L^2(\mathbb {R}^{n}\setminus 2Q_{j})} \\
					&  \leq \Big[
					\int_{\mathbb {R}^{n}\setminus 2Q_{j}} \int_{0}^{\infty} s \Big| \sum_{j}
					\int_{[0,r_j^{\alpha}]^m} L^{m+1} e^{-(|\textbf{u}|+s)L} b_{j}(x) d\textbf{u}
					\Big|^2	ds	dx
					\Big]^{1/2}\\
					&= \Big(\int_{\mathbb R^n\setminus 2Q_j} \|\sum_{j}F_j(\cdot,x) \|_{L^{2}(0,+\infty)}^2 dx\Big)^{1/2},
				\end{aligned}
			\end{equation}
			where we denote
			\begin{align*}
				F_j(s, x):= \sqrt{s}
				\int_{[0,r_j^{\alpha}]^m} L^{m+1} e^{-(|\textbf{u}|+s)L} b_{j}(x) d\textbf{u}.
			\end{align*}
			Using the Minkowski's inequality twice we have
			\begin{equation}\label{eq: Proof of sum of S_L 2}
				\begin{aligned}
					&\qquad\Big(\int_{\mathbb R^n\setminus 2Q_j} \|\sum_{j}F_j(\cdot,x) \|_{L^{2}(0,+\infty)}^2 dx\Big)^{1/2}\\
					&\leq \Big(
					\int_{\mathbb {R}^{n}\setminus 2Q_{j}} \Big( \sum_{j}
					\left\|
					F_j(\cdot, x)
					\right\|_{L^{2}((0,+\infty))} \Big)^2
					dx
					\Big)^{1/2}\\
					&\leq \sum_{j}  \Big(
					\int_{\mathbb {R}^{n}\setminus 2Q_{j}}
					\left\|
					F_j(\cdot,x)
					\right\|_{L^{2}((0,+\infty))}^2
					dx
					\Big)^{1/2}.
				\end{aligned}
			\end{equation}
			Since
			\begin{equation}\label{eq: Proof of sum of S_L 3}
				\begin{aligned}
					\|
					F_j(\cdot,x)
					\|_{L^{2}((0,+\infty))}^2 &= \int_{0}^{\infty} s \Big|
					\int_{[0,r_j^{\alpha}]^m} L^{m+1} e^{-(|\textbf{u}|+s)L} b_{j}(x) d\textbf{u}
					\Big|^2	ds\\
					&=\Big|
					S_{L}\Big((I-e^{-r_{j}^{\alpha} L} )^{m} b_j\Big)
					\Big|^2,
				\end{aligned}
			\end{equation}
			by combining \eqref{eq: Proof of sum of S_L 1}, \eqref{eq: Proof of sum of S_L 2} and \eqref{eq: Proof of sum of S_L 3}, we obtain \eqref{eq: sum of S_L}.

			Next we estimate \( \|S_{L}((I-e^{-r_{j}^{\alpha} L} )^{m} b_j) \|_{L^2(\mathbb {R}^{n}\setminus 2Q_{j})} \).
			We decompose \( \mathbb{R}^n \setminus 2Q_j \) into a union of annuli, i.e.,
			\begin{align*}
				\mathbb{R}^n \setminus 2Q_j = \bigcup_{k=2}^{+\infty} (2^{k}Q_{j}\setminus 2^{k-1}  Q_j)=: \bigcup_{k=2}^{+\infty}  S_{k}(Q_{j}).
			\end{align*}
			Denote \(G_{j}(s,x)\) by \(F_j(s,x)/\sqrt{s}\), i.e.,
			\begin{align*}
				G_j(s,x)= 	\int_{[0,r_j^{\alpha}]^m} L^{m+1} e^{-(|\textbf{u}|+s)L} b_{j}(x) d\textbf{u}.
			\end{align*}	We write
			\begin{align*}
				\quad	\|S_{L}((I-e^{-r_{j}^{\alpha} L} )^{m} b_j) \|_{L^2(\mathbb {R}^{n}\setminus 2Q_{j})}
				&=\Big(
				\int_{\mathbb {R}^{n}\setminus 2Q_{j}} \int_{0}^{\infty} |F_j(s,x)|^2	ds	dx
				\Big)^{1/2}\\
				&=\Big(
				\int_{0}^{\infty} s \int_{\mathbb {R}^{n}\setminus 2Q_{j}}   \left|
				G_j(s,x)
				\right|^2
				dx
				ds
				\Big)^{1/2}\\
				&= \Big( \int_{0}^{\infty} s \sum_{k}^{+\infty} \int_{2^{k+1}Q_{j}\setminus 2^k  Q_j} \left|  G_j(s,x) \right|^2  dx
				ds\Big)^{1/2}\\
				&=\Big( \int_{0}^{\infty} s \sum_{k}^{+\infty} \left\|  G_j(s,\cdot) \right\|_{L^2(S_{k}(Q_j)) }^2ds\Big)^{1/2}.
			\end{align*}

			The Minkowski's inequality gives
			\begin{align*}
				\quad   \|  G_j(s,\cdot) \|_{L^2(S_{k}(Q_j)) }^2
				&=
				\Big\|  	\int_{[0,r_j^{\alpha}]^m} L^{m+1} e^{-(|\textbf{u}|+s)L} b_{j} d\textbf{u} \Big\|_{L^2(S_{k}(Q_j)) }^2 \\
				&\leq
				\Big(
				\int_{[0,r_j^{\alpha}]^m}
				\left\|
				L^{m+1} e^{-(|\textbf{u}|+s)L} b_{j}
				\right\|_{L^2(S_{k}(Q_j)) }  d\textbf{u}
				\Big)^2\\
				&=: (G_{j,k}(s))^2.
			\end{align*}
			
			By applying Lemma \ref{lem: T_t: L^{p_0,1} to L^q} and the fact that
			\[\max\Big\{\big(\frac{a}{b}\big)^{n/p_0},\big(\frac{a}{b}\big)^{n}\Big\}\leq \max\Big\{\big(\frac{a}{b}\big)^{n/2},\big(\frac{a}{b}\big)^{n}\Big\}\]
			with $1<p_0<2$ and $a,b>0$, we have
			\begin{equation}\label{eq: L^2(S_k(Q_j)) norm}
				\begin{aligned}
					&\qquad	\|
					L^{m+1} e^{-(|\textbf{u}|+s)L} b_{j}
					\|_{L^2(S_{k}(Q_j)) } \\
					& \lesssim
					(|\textbf{u}|+s)^{-(m+1)}
					\max\Big\{\Big(\f{r_j}{(|\textbf{u}|+s)^{1/\alpha}}\Big)^{n/2}, \Big(\f{r_j}{(|\textbf{u}|+s)^{1/\alpha}}\Big)^{n}\Big\}\\
					&\quad \cdot\Big(1+\f{(|\textbf{u}|+s)^{1/\alpha}}{2^kr_j}\Big)^{n/2} \Big(1+\f{2^kr_j}{(|\textbf{u}|+s)^{1/\alpha}}\Big)^{-n-\epsilon}
					|S_{k}(Q_{j})|^{1/2} |Q_{j}|^{-1/p_0} \|b_j\|_{L^{p_{0}}}\\
					&\leq(|\textbf{u}|+s)^{-(m+1)}
					\max\Big\{\Big(\f{r_j}{(|\textbf{u}|+s)^{1/\alpha}}\Big)^{n/2}, \Big(\f{r_j}{(|\textbf{u}|+s)^{1/\alpha}}\Big)^{n}\Big\}\\
					&\quad  \cdot\Big(1+\f{(|\textbf{u}|+s)^{1/\alpha}}{2^kr_j}\Big)^{n/2} \Big(1+\f{2^kr_j}{(|\textbf{u}|+s)^{1/\alpha}}\Big)^{-n-\epsilon}
					2^{kn/2}|Q_{j}|^{1/2-1/p_0} \|b_j\|_{L^{p_{0},1}},
				\end{aligned}
			\end{equation}
			where the last equality follows from \(b_j=\chi_{E}\chi_{Q_{j}} \) and \(|S_k(Q_j)|\leq 2^{kn}|Q_j|\).
			
			Substituting \eqref{eq: L^2(S_k(Q_j)) norm} into \(G_{j,k}(s)\), we have
			\begin{align*}
				G_{j,k}(s) & \lesssim  2^{kn/2}|Q_{j}|^{\frac{1}{2}-\frac{1}{p_0}} \|b_j\|_{L^{p_{0},1}}
				\int_{[0,r_j^{\alpha}]^m}
				(|\textbf{u}|+s)^{-(m+1)}
				\max\Big\{\Big(\f{r_j}{(|\textbf{u}|+s)^{1/\alpha}}\Big)^{n/2},\\
				&\quad  \Big(\f{r_j}{(|\textbf{u}|+s)^{1/\alpha}}\Big)^{n}\Big\}
				\Big(1+\f{(|\textbf{u}|+s)^{1/\alpha}}{2^kr_j}\Big)^{n/2} \Big(1+\f{2^kr_j}{(|\textbf{u}|+s)^{1/\alpha}}\Big)^{-n-\epsilon}
				d\textbf{u}
			\end{align*}

			Therefore, we get
			\begin{align*}
				& \qquad
				\Big[ \int_{0}^{+\infty}
				s
				\sum_{k=2}^{+\infty}
				(G_{j,k}(s))^2
				ds
				\Big]^{1/2}
				\leq \Big[   \sum_{k=2}^{+\infty} \int_{0}^{+\infty}
				s
				(G_{j,k}(s))^2
				ds
				\Big]^{1/2}\\
				&\leq |Q_{j}|^{\frac{1}{2}-\frac{1}{p_0}} \|b_j\|_{L^{p_{0},1}} \Big [  \sum_{k=2}^{+\infty} \int_{0}^{+\infty}
				s 2^{kn}
				\Big [
				\int_{[0,r_j^{\alpha}]^m}
				(|\textbf{u}|+s)^{-(m+1)}
				\max\Big\{\Big(\f{r_j}{(|\textbf{u}|+s)^{1/\alpha}}\Big)^{n/2},\\
				&\quad  \Big(\f{r_j}{(|\textbf{u}|+s)^{1/\alpha}}\Big)^{n}\Big\}
				\Big(1+\f{(|\textbf{u}|+s)^{1/\alpha}}{2^kr_j}\Big)^{n/2} \Big(1+\f{2^kr_j}{(|\textbf{u}|+s)^{1/\alpha}}\Big)^{-n-\epsilon}d\textbf{u}
				\Big ]^2
				ds
				\Big ]^{1/2}.
			\end{align*}
			
			By decomposing the integral \(\int_{0}^{+\infty}\) as \(\int_{0}^{+\infty}=\int_{0}^{r_j^{\alpha}} + \int_{r_j^{\alpha}}^{(2^kr_{j})^{\alpha}} +\int_{(2^kr_j)^{\alpha}}^{+\infty}\), we obtain		
			\begin{align*}
				\Big[ \int_{0}^{+\infty}
				s
				\sum_{k=2}^{+\infty}
				(G_{j,k}(s))^2
				ds
				\Big]^{1/2}
				\leq |Q_j|^{\frac{1}{2}-\frac{1}{p_0}}  \|b_j\|_{L^{p_0,1}}  (\text{Term}\mathrm{III_1}+\text{Term}\mathrm{III_2}+\text{Term}\mathrm{III_3})^{1/2},
			\end{align*}
			where
			\begin{align*}
				\text{Term}\mathrm{III_1} &:=    \sum_{k=2}^{+\infty} 2^{kn}  \int_{0}^{r_j^{\alpha}}
				s
				\Big [
				\int_{[0, r_j^{\alpha}]^m}
				(|\textbf{u}|+s)^{-(m+1)}
				\max\Big\{\Big(\f{r_j}{(|\textbf{u}|+s)^{1/\alpha}}\Big)^{n/2},\\
				&\quad  \Big(\f{r_j}{(|\textbf{u}|+s)^{1/\alpha}}\Big)^{n}\Big\}
				\Big(1+\f{(|\textbf{u}|+s)^{1/\alpha}}{2^kr_j}\Big)^{n/2} \Big(1+\f{2^kr_j}{(|\textbf{u}|+s)^{1/\alpha}}\Big)^{-n-\epsilon}
				d\textbf{u}
				\Big ]^2
				ds,
			\end{align*}
			\begin{align*}
				\text{Term}\mathrm{III_2} &:=  \sum_{k=2}^{+\infty} 2^{kn} \int_{r_j^{\alpha}}^{(2^kr_{j})^{\alpha}}
				s
				\Big [
				\int_{[0, r_j^{\alpha}]^m}
				(\|\textbf{u}\|+s)^{-(m+1)}
				\max\Big\{\Big(\f{r_j}{(|\textbf{u}|+s)^{1/\alpha}}\Big)^{n/2},\\
				&\quad  \Big(\f{r_j}{(|\textbf{u}|+s)^{1/\alpha}}\Big)^{n}\Big\}
				\Big(1+\f{(|\textbf{u}|+s)^{1/\alpha}}{2^kr_j}\Big)^{n/2} \Big(1+\f{2^kr_j}{(|\textbf{u}|+s)^{1/\alpha}}\Big)^{-n-\epsilon}
				d\textbf{u}
				\Big ]^2
				ds,
			\end{align*}
			and
			\begin{align*}
				\text{Term}\mathrm{III_3} &:=  \sum_{k=2}^{+\infty} 2^{kn}  \int_{(2^kr_{j})^{\alpha}}^{+\infty}
				s
				\Big [
				\int_{[0, r_j^{\alpha}]^m}
				(|\textbf{u}|+s)^{-(m+1)}
				\max\Big\{\Big(\f{r_j}{(|\textbf{u}|+s)^{1/\alpha}}\Big)^{n/2},\\
				&\quad  \Big(\f{r_j}{(|\textbf{u}|+s)^{1/\alpha}}\Big)^{n}\Big\}
				\Big(1+\f{(|\textbf{u}|+s)^{1/\alpha}}{2^kr_j}\Big)^{n/2} \Big(1+\f{2^kr_j}{(|\textbf{u}|+s)^{1/\alpha}}\Big)^{-n-\epsilon} d\textbf{u}
				\Big ]^2
				ds.
			\end{align*}

			Next, we prove that the above three terms are convergent.
			
			\textbf{Estimate of \(\text{Term}\mathrm{III_1}\):}
			Since \((|\textbf{u}|+s)^{1/\alpha}\lesssim r_j\), we have
			\begin{align*}
				\text{Term}\mathrm{III_1} &\lesssim   \sum_{k=2}^{+\infty} 2^{kn}  \int_{0}^{r_j^{\alpha}}
				s
				\Big [
				\int_{[0, r_j^{\alpha}]^m}
				(|\textbf{u}|+s)^{-(m+1)-n/\alpha} r_{j}^{n}\\
				&\qquad\qquad\qquad\qquad\quad\cdot\Big(1+\f{(|\textbf{u}|+s)^{1/\alpha}}{2^kr_j}\Big)^{n/2} \Big(1+\f{2^kr_j}{(|\textbf{u}|+s)^{1/\alpha}}\Big)^{-n-\epsilon}
				d\textbf{u}
				\Big ]^2
				ds\\
				&\lesssim  \sum_{k=2}^{+\infty} 2^{kn}  \int_{0}^{r_j^{\alpha}}
				s
				\Big [
				\int_{[0, r_j^{\alpha}]^m}
				(|\textbf{u}|+s)^{-(m+1)-n/\alpha} r_{j}^{n}
				\Big(1+\f{2^kr_j}{(|\textbf{u}|+s)^{1/\alpha}}\Big)^{-n-\epsilon}
				d\textbf{u}
				\Big ]^2
				ds\\
				&\lesssim  \sum_{k=2}^{\infty} 2^{-nk-2\epsilon k} r_j^{-2\epsilon} \int_{0}^{r_j^{\alpha}} s \Big[ \int_{[0, r_j^{\alpha}]^m} (|\textbf{u}|+s)^{-(m+1)+\epsilon/\alpha}  d\textbf{u} \Big]^2 ds.
			\end{align*}
			We write
			\begin{align*}
				\int_{0}^{r_j^{\alpha}} s \Big[ \int_{[0, r_j^{\alpha}]^m} (|\textbf{u}|+s)^{-(m+1)+\epsilon/\alpha}  d\textbf{u} \Big]^2 ds= \int_{0}^{r_j^{\alpha}} s \Big[ \int_{[0, r_j^{\alpha}]^m}
				\Big(	\frac{1}{(|\textbf{u}|+s)^{1-\frac{\epsilon}{\alpha (m+1)}}} \Big)^{m+1}
				d\textbf{u} \Big]^2 ds,
			\end{align*}
			then we obtain
			\begin{align*}
				\int_{0}^{r_j^{\alpha}} s \Big[ \int_{[0, r_j^{\alpha}]^m} (|\textbf{u}|+s)^{-(m+1)+\epsilon/\alpha}  d\textbf{u} \Big]^2 ds
				&\leq \int_{0}^{r_j^{\alpha}} s^{-1+\frac{2\epsilon}{\alpha (m+1)}}  \Big( \int_{0}^{ r_j^{\alpha}}
				\frac{1}{u^{1-\frac{\epsilon}{\alpha (m+1)}}}
				du \Big)^{2m} ds\\
				&=r_j^{(\frac{2m}{m+1}+\frac{2}{m+1})\epsilon}\\
				&=r_j^{2\epsilon}.
			\end{align*}
			Therefore, \(\text{Term}\mathrm{III_1} \lesssim 1\).

		\textbf{Estimate of \(\text{Term}\mathrm{III_2}\):}
		For \(s\in (r_j^{\alpha},(2^kr_j)^{\alpha})\), we have \(s\leq|\textbf{u}|+s\leq mr_{j}^{\alpha}+s\lesssim s\), that is \(|\textbf{u}|+s\approx s\).
		Since \[\frac{r_j}{(|\textbf{u}|+s)^{1/\alpha}} <1 \text{~ and ~}r_j<s^{1/\alpha}<2^kr_j,\]
		$\text{Term}\mathrm{III_2}$ can be controlled by
		\begin{align*}
						&\quad\sum_{k=2}^{+\infty} 2^{kn} \int_{r_j^{\alpha}}^{(2^kr_{j})^{\alpha}}
			s
			\Big [
			\int_{[0, r_j^{\alpha}]^m}
			(|\textbf{u}|+s)^{-(m+1)-n/(2\alpha)} r_j^{n/2}
			\Big(1+\f{(|\textbf{u}|+s)^{1/\alpha}}{2^kr_j}\Big)^{n/2}\\
			&\qquad\qquad\qquad\qquad\qquad \cdot \Big(1+\f{2^kr_j}{(|\textbf{u}|+s)^{1/\alpha}}\Big)^{-n-\epsilon}
			d\textbf{u}
			\Big ]^2
			ds\\
			&\approx  \sum_{k=2}^{+\infty} 2^{kn} \int_{r_j^{\alpha}}^{(2^kr_{j})^{\alpha}}
			s
			\Big [
			\int_{[0, r_j^{\alpha}]^m}
			s^{-(m+1)-n/(2\alpha)} r_j^{n/2}
			\Big(1+\f{s^{1/\alpha}}{2^kr_j}\Big)^{n/2} \Big(1+\f{2^kr_j}{s^{1/\alpha}}\Big)^{-n-\epsilon}
			d\textbf{u}
			\Big ]^2
			ds\\
			&\lesssim 	\sum_{k=2}^{+\infty} 2^{-kn-2k\epsilon}
			r_j^{-n-2\epsilon+2\alpha m} \int_{r_j^{\alpha}}^{(2^kr_{j})^{\alpha}}
			s^{-1-2m+n/\alpha+2\epsilon/\alpha}
			ds.
			\end{align*}
		By calculating,
		\begin{align*}
			r_j^{-n-2\epsilon+2\alpha m} \int_{r_j^{\alpha}}^{(2^kr_{j})^{\alpha}}
			s^{-1-2m+n/\alpha+2\epsilon/\alpha}
			ds\lesssim
			2^{k(-2m\alpha+n+2\epsilon)},
		\end{align*}
		which implies
		\begin{align*}
			\text{Term}\mathrm{III_2} \lesssim 1.
		\end{align*}

		\textbf{Estimate of \(\text{Term}\mathrm{III_3}\):}
		For \(s\in ((2^kr_j)^{\alpha},\infty)\), we still have \(s\leq|\textbf{u}|+s\leq mr_{j}^{\alpha}+s\lesssim s\), i.e., \(|\textbf{u}|+s\approx s\). Using the equivalence and the estimate \(1+\frac{s^{1/\alpha}}{2^{k}r_j}\leq \frac{2s^{1/\alpha}}{2^{k}r_j}\) we have
		\begin{align*}
			\text{Term}\mathrm{III_3} &\approx  \sum_{k=2}^{+\infty} 2^{kn}  \int_{(2^kr_{j})^{\alpha}}^{+\infty}
			s
			\Big [
			\int_{[0, r_j^{\alpha}]^m}
			s^{-m-1-n/(2\alpha)} r_j^{n/2}
			\Big(1+\f{s^{1/\alpha}}{2^kr_j}\Big)^{n/2} \Big(1+\f{2^kr_j}{s^{1/\alpha}}\Big)^{-n-\epsilon} d\textbf{u}
			\Big ]^2
			ds\\
			&\lesssim  \sum_{k=2}^{+\infty}  \int_{(2^kr_{j})^{\alpha}}^{+\infty}
			s
			\Big [
			\int_{[0, r_j^{\alpha}]^m}
			s^{-m-1}
			\Big(1+\f{2^kr_j}{s^{1/\alpha}}\Big)^{-n-\epsilon} d\textbf{u}
			\Big ]^2
			ds\\
			&\leq  \sum_{k=2}^{+\infty} r_j^{2m\alpha} \int_{(2^kr_{j})^{\alpha}}^{+\infty}
			s^{-1-2m}
			ds.
		\end{align*}
		Since
		\begin{align*}
			\int_{(2^kr_{j})^{\alpha}}^{+\infty}
			s^{-1-2m}
			ds\lesssim 2^{-2\alpha mk} r_j^{-2m\alpha},
		\end{align*}
		we obtain
		\begin{align*}
			\text{Term}\mathrm{III_3}\lesssim  \sum_{k=2}^{+\infty} 2^{-2\alpha mk}\lesssim 1.
		\end{align*}

		Hence, we obtain\begin{align*}\label{eq:term3 L2 estimate}
			\|S_{L}(I-e^{-r_{j}^{\alpha} L} )^{m} b_j  \|_{L^2(\mathbb {R}^{n}\setminus 2Q_{j})}  & \lesssim |Q_{j}|^{1/2-1/p_{0}} \|b_{j}\|_{L^{p_{0},1}}.
		\end{align*}
		
		Since
		\begin{align*}
			|Q_j|\leq \lambda^{-p_0} \int_{Q_j} \chi_{E}^{p_0} (x) dx\leq 2^n |Q_j|,
		\end{align*}
		and
		\begin{align*}
			| \bigcup_{j} Q_{j}  | =\sum_{j} |Q_j|  \lesssim \lambda^{-p_0}  \|\chi_{E}\|_{L^{p_0,1}}^{p_0},
		\end{align*}
		we have
		\begin{equation}\label{eq:term3 L2 estimate2}
			\begin{aligned}
				\frac{1}{\lambda^2} \sum_{j}
				\left\| S_{L} (I-e^{r_{j}^\alpha L})^m  b_{j}
				\right\|_{L^2(\mathbb{R}^n\setminus 2Q_{j})}^2
				& \leq  \frac{1}{\lambda^2} \sum_{j}  |Q_{j}|^{1-2/p_{0}} \|b_j\|_{L^{p_{0},1}}^2\\
				&\leq \frac{1}{\lambda^2} \sum_{j} |Q_j| \Big( \frac{1}{|Q_j|} \int_{Q_j} \chi_{E}^{p_0} (x) dx \Big)^{2/p_0}\\
				&\lesssim \lambda^{-p_0}\| \chi_{E}\|_{L^{p_0,1}}^{p_0},
			\end{aligned}
		\end{equation}
		where the second inequality follows from the setting that \( b_{j}=\chi_{E}\chi_{Q_{j}} \) and \( \|b_j\|_{L^{p_0,1}}=\|b_j\|_{L^{p_0}} \), and the third inequality is from that\[ \frac{1}{ |Q_j| } \int_{Q_j} \chi_{E}^{p_0} (x) dx\leq 2^n \lambda^{p_0} .\]
		By \eqref{eq: S_L h_2 estimate}, \eqref{eq: sum of S_L} and \eqref{eq:term3 L2 estimate2}, we have
		\begin{align*}
			&\qquad \Big|
			\big\{x\in \mathbb R^n: \big|
			S_L\big(\sum_{j}(I-e^{-r_j^\alpha L})^m b_j\big)(x)
			\big| >\lambda
			\big\}
			\Big| \\
			&  \lesssim \lambda^{-p_0} \|\chi_E\|_{L^{p_0}(\mathbb{R}^n)}^{p_0} + \lambda^{-2} \Big\| S_L\big(
			\sum_{j}(I-e^{-r_j^\alpha L})^m b_j \big)
			\Big\|_{L^2(\mathbb R^n\setminus 2Q_j)}^2\\
			&\lesssim  \lambda^{-p_0}\| \chi_{E}\|_{L^{p_0,1}}^{p_0}.
		\end{align*}
		Hence, the proof of \eqref{eq:S_Lb} is complete, and \eqref{eq:aim1} follows.

		\section{Endpoint Estimates of \texorpdfstring{$\mathcal{M}(L)$}{M(L)}}\label{sec proof of ML}
		Recall that the functional calculus of Laplace transform type \(\mathcal M(L)\) is defined by
		\begin{align*}
			\mathcal{M} (L) =\int_{0}^{\infty} m(t)Le^{-tL}dt,
		\end{align*}
		where \(m: [0,\infty)\rightarrow \mathbb C\) is a bounded function. To prove the boundedness of \(\mathcal M(L)\) from \(L^{p_0,1}\) to \(L^{p_0,\infty}\), it is still sufficient to prove that \(\mathcal M(L)\) is of restricted weak type \((p_0,p_0)\), i.e., for all \(\lambda>0\), and for any measurable set \(E\) of finite measure,
		\begin{align*}
			|\{x\in \mathbb R^n: |\mathcal M(L)\chi_{E}(x)|>\lambda\}| \lesssim \lambda^{-p_0} \|\chi_{E} \|_{L^{p_0}(\mathbb R^n)}^{p_0}.
		\end{align*}
		
		We still make the Calder\'on--Zygmund decomposition of \(\chi_{E}\) and use the same notation of \(h_1\) and \(h_2\) as in Section \ref{section: square operators}.  Therefore, \(\chi_{E}=g+h_1+h_2\) and we write
		\begin{align*}
			|\{x\in \mathbb R^n:|\mathcal M(L)\chi_{E}(x)|>\lambda\}| &\leq |\{x\in \mathbb R^n:|\mathcal M(L)g(x)|>\frac{\lambda}{3}\}|\\
			&\quad+  |\{x\in \mathbb R^n:|\mathcal M(L)h_1(x)|>\frac{\lambda}{3}\}|\\
			&\quad+ |\{x\in \mathbb R^n:|\mathcal M(L)h_2(x)|>\frac{\lambda}{3}\}|.
		\end{align*}
		
		The estimate
		\begin{align*}
			|\{x\in \mathbb R^n:|\mathcal M(L)g(x)|>\frac{\lambda}{3}\}| \lesssim \lambda^{-p_0} \|\chi_{E}\|_{L^{p_0}(\mathbb R^n)}^{p_0}
		\end{align*}
		is a direct result of Chebyshev's inequality, the \(L^2\) boundedness of \(\mathcal M(L)\) and the property that \(\|g\|_{L^{p_0}(\mathbb R^n)}\leq \|\chi_{E}\|_{L^{p_0}(\mathbb R^n)}\).
		
		Using the result \eqref{eq:h1 result}, the Chebyshev's inequality, Calder\'on--Zygmund decomposition and the \(L^2\) boundedness of \(\mathcal M(L)\), the estimate
		\begin{align*}
			|\{x\in \mathbb R^n:|\mathcal M(L)h_1(x)|>\frac{\lambda}{3}\}| \lesssim \lambda^{-p_0} \|\chi_{E}\|_{L^{p_0}(\mathbb R^n)}^{p_0}
		\end{align*}
		follows directly.

		In what follows, we devote to estimate the term \(|\{x\in \mathbb R^n:|\mathcal M(L)h_2(x)|>\frac{\lambda}{3}\}|\). By substituting \(h_2=\sum_{j}(I-e^{-r_j^{\alpha}L})^mb_j\) into the expression of \(\mathcal M(L)\) we have
		\begin{align*}
			\mathcal M(L)h_2(x) &= \int_{0}^{\infty} m(t) Le^{-tL}  \Big(\sum_{j}(I-e^{-r_j^{\alpha}L})^mb_j\Big)(x) dt\\
			&\leq\sum_{j} \int_{0}^{\infty}  m(t) Le^{-tL}  \Big((I-e^{-r_j^{\alpha}L})^mb_j\Big)(x) dt.
		\end{align*}
		
		Denote the kernel of  the operator \(\frac{m(t)}{t}tLe^{-tL}\) by \(T_{m(t)}(x,y)\). Then we write
		\begin{align*}
			\frac{m(t)}{t} tLe^{-tL} \Big(\sum_{j} (I-e^{-r_j^{\alpha}L})^mb_j\Big) (x) = \int_{\mathbb R^n} T_{m(t)}(x,y)  \Big(\sum_{j}
			(I-e^{-r_j^{\alpha}L})^mb_j \Big)(y) dy.
		\end{align*}
		The Minkowski's inequality gives
		\begin{align*}
			\|\mathcal{M}(L)h_2 \|_{L^2(\mathbb {R}^{n}\setminus 2Q_{j})}
			&\leq \Big\| \sum_{j} \int_{0}^{\infty}  m(t) Le^{-tL}  \Big((I-e^{-r_j^{\alpha}L})^mb_j\Big) dt\Big\|_{L^2(\mathbb {R}^{n}\setminus 2Q_{j})}\\
			&\leq \sum_{j}  \Big\| \int_{0}^{\infty}  m(t) Le^{-tL}  \Big((I-e^{-r_j^{\alpha}L})^mb_j\Big) dt\Big\|_{L^2(\mathbb {R}^{n}\setminus 2Q_{j})}\\
						&\leq \sum_{j} \int_{0}^{\infty} |m(t)| \Big\| Le^{-tL}  \Big((I-e^{-r_j^{\alpha}L})^mb_j\Big) \Big\|_{L^2(\mathbb {R}^{n}\setminus 2Q_{j})}dt.
		\end{align*}
		By decomposing \(\mathbb {R}^{n}\setminus 2Q_{j}=\cup_{k=2}^{\infty} S_{k}(Q_j)\) we have
				\begin{align*}
			&\qquad \int_{0}^{\infty} |m(t)| \Big\| Le^{-tL}  \Big((I-e^{-r_j^{\alpha}L})^mb_j\Big) \Big\|_{L^2(\mathbb {R}^{n}\setminus 2Q_{j})}dt\\
			&=  \int_{0}^{\infty} |m(t)| \Big[
			\sum_{k=2}^{\infty}
			\int_{S_{k}(Q_j)}
			\Big|
			\int_{[0, r_j^{\alpha}]^m} L^{m+1} e^{-(|\textbf{u}|+t)L} b_{j}(x) d\textbf{u}
			\Big|^{2} dx
			\Big]^{1/2}  dt\\
			&\leq \int_{0}^{\infty} |m(t)| \sum_{k=2}^{\infty} 	\Big\|
			\int_{[0, r_j^{\alpha}]^m} L^{m+1} e^{-(|\textbf{u}|+t)L} b_{j} d\textbf{u}
			\Big\|_{L^{2}(S_{k}(Q_j))}  dt\\
			&\leq \sum_{k=2}^{\infty}  \int_{0}^{\infty} |m(t)| 	\Big\|
			\int_{[0, r_j^{\alpha}]^m} L^{m+1} e^{-(|\textbf{u}|+t)L} b_{j} d\textbf{u}
			\Big\|_{L^{2}(S_{k}(Q_j))}  dt.
		\end{align*}
		Using the Minkowski's inequality again and the boundedness of \(m(t)\), we have
		\begin{equation}\label{eq: mt estimate}
			\begin{aligned}
				&\qquad \sum_{k=2}^{\infty}  \int_{0}^{\infty} |m(t)| 	\Big\|
				\int_{[0, r_j^{\alpha}]^m} L^{m+1} e^{-(|\textbf{u}|+t)L} b_{j} d\textbf{u}
				\Big\|_{L^{2}(S_{k}(Q_j))}  dt\\
				&\leq \sum_{k=2}^{\infty}  \int_{0}^{\infty} |m(t)|
				\int_{[0, r_j^{\alpha}]^m} 	\left\| L^{m+1} e^{-(|\textbf{u}|+t)L} b_{j}\right\|_{L^{2}(S_{k}(Q_j))}   d\textbf{u}
				dt\\
				&\lesssim \sum_{k=2}^{\infty}  \int_{0}^{\infty}
				\int_{[0, r_j^{\alpha}]^m} 	\left\| L^{m+1} e^{-(|\textbf{u}|+t)L} b_{j}\right\|_{L^{2}(S_{k}(Q_j))}   d\textbf{u}
				dt.
				\end{aligned}
		\end{equation}
		By substituting the result \eqref{eq: L^2(S_k(Q_j)) norm} into \eqref{eq: mt estimate} we get
		\begin{align*}
			&\qquad \|\mathcal{M}(L)h_2 \|_{L^2(\mathbb {R}^{n}\setminus 2Q_{j})}\\
			&\lesssim \sum_{k=2}^{\infty}  \int_{0}^{\infty}
			\int_{[0, r_j^{\alpha}]^m} 	\left\| L^{m+1} e^{-(|\textbf{u}|+t)L} b_{j}\right\|_{L^{2}(S_{k}(Q_j))}
			d\textbf{u}
			dt\\
			&\lesssim  |Q_j|^{\frac{1}{2}-\frac{1}{p_0}}  \|b_j\|_{L^{p_{0},1}} \sum_{k=2}^{\infty} 2^{\frac{kn}{2}}  			\int_{0}^{\infty}
			\int_{[0, r_j^{\alpha}]^m}
			(|\textbf{u}|+t)^{-(m+1)}	
			\max\Big\{\Big(\f{r_j}{(|\textbf{u}|+t)^{1/\alpha}}\Big)^{n/2}, \\
			&\qquad \Big(\f{r_j}{(|\textbf{u}|+t)^{1/\alpha}}\Big)^{n}\Big\}
			\Big(1+\f{(|\textbf{u}|+t)^{1/\alpha}}{2^kr_j}\Big)^{n/2} \Big(1+\f{2^kr_j}{(|\textbf{u}|+t)^{1/\alpha}}\Big)^{-n-\epsilon} d\textbf{u}
			dt\\
			&=: |Q_j|^{\frac{1}{2}-\frac{1}{p_0}}  \|b_j\|_{L^{p_{0},1}}  (\text{Term}\mathrm{IV_1}+\text{Term}\mathrm{IV_2}+\text{Term}\mathrm{IV_3}),
		\end{align*}
		where
		\begin{align*}
			\text{Term}\mathrm{IV_1}&:=\sum_{k=2}^{\infty} 2^{\frac{kn}{2}}  \int_{0}^{r_j^{\alpha}}
			\int_{[0, r_j^{\alpha}]^m}
			(|\textbf{u}|+t)^{-(m+1)}
			\max\Big\{\Big(\f{r_j}{(|\textbf{u}|+t)^{1/\alpha}}\Big)^{n/2},\\
			&\qquad 	 \Big(\f{r_j}{(|\textbf{u}|+t)^{1/\alpha}}\Big)^{n}\Big\}
			\Big(1+\f{(|\textbf{u}|+t)^{1/\alpha}}{2^kr_j}\Big)^{n/2} \Big(1+\f{2^kr_j}{(|\textbf{u}|+t)^{1/\alpha}}\Big)^{-n-\epsilon} d\textbf{u}
			dt,
		\end{align*}
		\begin{align*}
			\text{Term}\mathrm{IV_2}&:= \sum_{k=2}^{\infty} 2^{\frac{kn}{2}}  \int_{r_{j}^{\alpha}}^{(2^kr_{j})^{\alpha}}
			\int_{[0, r_j^{\alpha}]^m}
			(|\textbf{u}|+t)^{-(m+1)}
			\max\Big\{\Big(\f{r_j}{(|\textbf{u}|+t)^{1/\alpha}}\Big)^{n/2},\\
			&\qquad 	 \Big(\f{r_j}{(|\textbf{u}|+t)^{1/\alpha}}\Big)^{n}\Big\}
			\Big(1+\f{(|\textbf{u}|+t)^{1/\alpha}}{2^kr_j}\Big)^{n/2}\Big(1+\f{2^kr_j}{(|\textbf{u}|+t)^{1/\alpha}}\Big)^{-n-\epsilon} d\textbf{u}
			dt,
		\end{align*}
		and
		\begin{align*}
			\text{Term}\mathrm{IV_3}&:= \sum_{k=2}^{\infty} 2^{\frac{kn}{2}}  \int_{(2^kr_{j})^{\alpha}}^{\infty}
			\int_{[0, r_j^{\alpha}]^m}
			(|\textbf{u}|+t)^{-(m+1)}
			\max\Big\{\Big(\f{r_j}{(|\textbf{u}|+t)^{1/\alpha}}\Big)^{n/2},\\
			&\qquad 	 \Big(\f{r_j}{(|\textbf{u}|+t)^{1/\alpha}}\Big)^{n}\Big\}
			\Big(1+\f{(|\textbf{u}|+t)^{1/\alpha}}{2^kr_j}\Big)^{n/2} \Big(1+\f{2^kr_j}{(|\textbf{u}|+t)^{1/\alpha}}\Big)^{-n-\epsilon} d\textbf{u}
			dt.
		\end{align*}
		\textbf{Estimate of \(\text{Term}\mathrm{IV_1}\):} Since \(|\textbf{u}|+t\lesssim r_j^{\alpha} \), we have
		\begin{align*}
			\text{Term}\mathrm{IV_1}&=\sum_{k=2}^{\infty} 2^{\frac{kn}{2}}
			r_j^{n} \int_{0}^{r_j^{\alpha}}
			\int_{[0, r_j^{\alpha}]^m}
			(|\textbf{u}|+t)^{-(m+1)-n/\alpha}\Big(1+\f{(|\textbf{u}|+t)^{1/\alpha}}{2^kr_j}\Big)^{n/2}
			\\
			&\qquad\qquad\qquad\qquad\quad \cdot\Big(1+\f{2^kr_j}{(|\textbf{u}|+t)^{1/\alpha}}\Big)^{-n-\epsilon} d\textbf{u}
			dt\\
			&\lesssim \sum_{k=2}^{\infty} 2^{-k(\frac{n}{2}+\epsilon)}
			r_j^{-\epsilon} \int_{0}^{r_j^{\alpha}}
			\int_{[0, r_j^{\alpha}]^m}
			(|\textbf{u}|+t)^{-(m+1)+\epsilon/\alpha}
			d\textbf{u}
			dt.
		\end{align*}
		We write
		\begin{align*}
			\int_{0}^{r_j^{\alpha}} \int_{[0, r_j^{\alpha}]^m}
			(|\textbf{u}|+t)^{-(m+1)+\epsilon/\alpha}
			d\textbf{u}
			dt = \int_{0}^{r_j^{\alpha}} 	\int_{[0, r_j^{\alpha}]^m}
			\Big(
			\frac{1}{(|\textbf{u}| +t)^{1-\frac{\epsilon}{\alpha(m+1)}}}
			\Big)^{m+1}
			d\textbf{u}
			dt,
		\end{align*}
		to get
		\begin{align*}
			\int_{0}^{r_j^{\alpha}} \int_{[0, r_j^{\alpha}]^m}
			(|\textbf{u}|+t)^{-(m+1)+\epsilon/\alpha}
			d\textbf{u}
			dt &\leq \int_{0}^{r_j^{\alpha}}
			t^{-1+\frac{\epsilon}{\alpha(m+1)}}
			\Big(
			\int_{0}^{r_j^{\alpha}}
			u^{-1+\frac{\epsilon}{\alpha(m+1)}} du
			\Big)^{m}
			dt=r_j^{\epsilon}.
		\end{align*}
		Therefore, we obtain that \(\text{Term}\mathrm{IV_1}\lesssim 1\).

		\textbf{Estimate of \(\text{Term}\mathrm{IV_2}\):}
		We use the equivalence \(|\textbf{u}|+t\approx t\), then we have
		\begin{align*}
			\text{Term}\mathrm{IV_2}&\approx \sum_{k=2}^{\infty} 2^{\frac{kn}{2}}  \int_{r_{j}^{\alpha}}^{(2^kr_{j})^{\alpha}}
			\int_{[0, r_j^{\alpha}]^m}
			t^{-(m+1)-n/2\alpha} r_j^{n/2}
			\Big(1+\f{t^{1/\alpha}}{2^kr_j}\Big)^{n/2}\Big(1+\f{2^kr_j}{t^{1/\alpha}}\Big)^{-n-\epsilon} d\textbf{u}
			dt\\
			&\lesssim \sum_{k=2}^{\infty} 2^{-\frac{kn}{2}-\epsilon k}
			r_j^{-n/2-\epsilon+\alpha m}
			\int_{r_{j}^{\alpha}}^{(2^kr_{j})^{\alpha}}
			t^{-(m+1)+n/2\alpha+\epsilon/\alpha}
			dt,
		\end{align*}
		and
		\begin{align*}
			\int_{r_{j}^{\alpha}}^{(2^kr_{j})^{\alpha}}
			t^{-(m+1)+n/2\alpha+\epsilon/\alpha}
			dt = \frac{2^{k(-m\alpha+n/2+\epsilon)}-1}{-m+n/\alpha+\epsilon/\alpha} r_j^{-\alpha m+n/2+\epsilon} .
		\end{align*}
		
		Therefore, we get
		\begin{align*}
			\text{Term}\mathrm{IV_2}\lesssim \sum_{k=2}^{\infty} 2^{-m\alpha k}
			\lesssim 1.
		\end{align*}
		
		\textbf{Estimate of \(\text{Term}\mathrm{IV_3}\):}
		We use the equivalence \(|\textbf{u}|+t\approx t\) again to get
		\begin{align*}
			\text{Term}\mathrm{IV_3} & \approx \sum_{k=2}^{\infty} 2^{\frac{kn}{2}}  \int_{(2^kr_{j})^{\alpha}}^{\infty}
			\int_{[0, r_j^{\alpha}]^m}
			t^{-(m+1)-n/2\alpha} r_j^{n/2}
			\Big(1+\f{t^{1/\alpha}}{2^kr_j}\Big)^{n/2} \Big(1+\f{2^kr_j}{t^{1/\alpha}}\Big)^{-n-\epsilon} d\textbf{u}
			dt\\
			&\lesssim  \sum_{k=2}^{\infty}  \int_{(2^kr_{j})^{\alpha}}^{\infty}
			\int_{[0, r_j^{\alpha}]^m}
			t^{-(m+1)}  \Big(1+\f{2^kr_j}{t^{1/\alpha}}\Big)^{-n-\epsilon} d\textbf{u}
			dt\\
			&\lesssim  \sum_{k=2}^{\infty}
			r_j^{\alpha m}
			\int_{(2^kr_{j})^{\alpha}}^{\infty} 	t^{-(m+1)}
			dt.
		\end{align*}
		Since
		\begin{align*}
			\int_{(2^kr_{j})^{\alpha}}^{\infty} 	t^{-(m+1)}
			dt = \frac{1}{m} 2^{-\alpha m k} r_{j}^{-\alpha m},
		\end{align*}
		we obtain
		\begin{align*}
			\text{Term}\mathrm{IV_3}\lesssim \sum_{k}^{\infty} 2^{-\alpha m k}\lesssim 1.
		\end{align*}
		
		It follows that
		\begin{align*}
			\qquad \|\mathcal{M}(L)h_2 \|_{L^2(\mathbb {R}^{n}\setminus 2Q_{j})}
			&\lesssim  \sum_{j}|Q_j|^{\frac{1}{2}-\frac{1}{p_0}}  \|b_j\|_{L^{p_{0},1}}.
		\end{align*}
		
		Similar to the calculations
		in \eqref{eq: S_L h_2 estimate} and \eqref{eq:term3 L2 estimate2}, we have
		\begin{align*}
			|\{x\in \mathbb R^n:|\mathcal M(L)h_2(x)|>\frac{\lambda}{3}\}| \lesssim \lambda^{-p_0} \|\chi_{E}\|_{L^{p_0}(\mathbb R^n)}^{p_0}.
		\end{align*}
		Therefore, this completes the proof of Theorem \ref{THEMML}.

		%\section*{Statements and Declarations}
		
	%	\noindent\textbf{Acknowledgments}
		\section*{Acknowledgments}
		%This paper is part of Xuejing Huo's PhD thesis. 
		The authors would like to thank The Anh Bui and Ji Li for introducing the topics and for valuable discussions and helpful suggestions. Xueting Han is supported by the Hefei Institute of Technology under grant numbers 2025KY61 and 2025AHGXZK40203. Xuejing Huo is supported by the the International Macquarie University Research Training Program (iMQRTP)  Scholarship.
        %The authors would like to thank Ji Li for his very helpful suggestions.
		
	%	\noindent\textbf{Ethical Approval}
		
	%	The declaration for ethical approval is not applicable.
		
%		\noindent\textbf{Competing interests}
		
	%	The authors declare no competing interest.

	%	\noindent\textbf{Authors'contributions}
		
	%	Xueting Han and Xuejing Huo wrote and reviewed
	%	the manuscript. 
		
	%	\noindent\textbf{Funding}

	%	\noindent\textbf{Availability of data and materials}
		
	%	Data sharing is not applicable to this article as no datasets were generated or analysed during the current study.

	\end{document}